\documentclass[article]{siamonline190516}
\usepackage{amsmath} 
\usepackage{verbatim}
\usepackage{textcomp}
\usepackage{amsfonts}
\usepackage{amssymb}
\usepackage{graphicx}
\usepackage{ragged2e}
\usepackage[margin=1in]{geometry}
\usepackage{xcolor}

\newcommand{\R}{\mathbb{R}}

\newcommand{\E}{\mathbb{E}}
\newcommand{\bpf}{\begin{proof}}
\newcommand{\epf}{\end{proof}}

\newtheorem{remark}{Remark}

\title{On Higher Order Drift and Diffusion Estimates for Stochastic SINDy \thanks{Submitted to the editor April 21, 2023, Resubmitted December 1, 2023}}

\author{Mathias Wanner\thanks{Department of Mechanical Engineering, 
		University of California, Santa Barbara (\email{mwanner@ucsb.edu},\email{mezic@ucsb.edu})}
	\and
	Dr. Igor Mezi\'{c} \footnotemark[2]
}

\begin{document}
	\maketitle
	\begin{abstract}
		The Sparse Identification of Nonlinear Dynamics (SINDy) algorithm can be applied to stochastic differential equations to estimate the drift and the diffusion function using data from a realization of the SDE. The SINDy algorithm requires sample data from each of these functions, which is typically estimated numerically from the data of the state. We analyze the performance of the previously proposed estimates for the drift and diffusion function to give bounds on the error for finite data. However, since this algorithm only converges as both the sampling frequency and the length of trajectory go to infinity, obtaining approximations within a certain tolerance may be infeasible. To combat this, we develop estimates with higher orders of accuracy for use in the SINDy framework. For a given sampling frequency, these estimates give more accurate approximations of the drift and diffusion functions, making SINDy a far more feasible system identification method.
	\end{abstract}
	
	

\section{Introduction}
For many dynamical systems, data may abundant while there remains no analytic models to describe the system. These systems may be too complex, may have too large a dimension, or may be too poorly understood to model using first principles. For these reasons, data driven modeling has become important for applications in science and engineering. There is a wide variety of system identification methods, ranging from classical methods, \cite{ljung1998system}, to Dynamic Mode Decomposition and Koopman operator methods, \cite{schmid2010dynamic,williams2015data,mezic2005spectral,wanner2022robust}, to neural networks \cite{kumpati1990identification,kosmatopoulos1995high} and many others. These methods vary in their their complexity, training methods, model sizes, and interpretability. Sparse Identification of Nonlinear Dynamics (SINDy) is a method which allows for some complexity (allowing nonlinear models over only linear ones) while the sparse solution promotes simple, interpretable models.

The SINDy algorithm, developed by Brunton et. al. \cite{brunton2016discovering} estimates the parameters of an ordinary differential equation from data. It does this by using a dictionary of functions and finding a sparse representation of the derivative in this dictionary. The data for the derivative can be obtained using finite differences of data from the state. For ODEs, the performance of this algorithm has been analyzed in \cite{zhang2019convergence}.

SINDy has several extensions and adaptations; it has also been extended to identify control systems \cite{brunton2016sparse,kaiser2018sparse}, adapted to systems with implicit solutions \cite{mangan2016inferring,kaheman2020sindy}, and formulated in ways to improve its robustness to noise \cite{fasel2022ensemble,messenger2021weak,messenger2021weak2}, to name a few. Additionally, different methods for computing the sparse solution have been proposed, including LASSO \cite{tibshirani1996regression}, the sequential thresholding presented in the original paper \cite{brunton2016discovering}.

The problem of system identification can be similarly posed for stochastic differential equations. Many systems, due to their complexity, seperation of timescales, or intrinsic randomness, lead to data that may be better approximated as a stochastic process. However, these systems may require more sophisticated tools of analysis \cite{friedrich2011approaching}. In order to identify an SDE, we need to estimate a diffusion function, which determines the nature of the random forcing, in addition to the drift, which represents the mean dynamics. In the context of single particle tracking, the diffusion constant is locally estimated using the mean square displacement \cite{qian1991single}, and the uncertainty can be quantified and compared against the Cramer-Rao bound \cite{michalet2012optimal,vestergaard2014optimal}.
	
The mean square displacement can be generalized for the estimation of SDEs, where the drift and diffusion functions may vary spatially. Local approximations for the drift and diffusion functions can be obtained from data using the Kramer-Moyal expansion, and can be used to estimate spatially varying parameters \cite{siegert1998analysis,friedrich2000extracting,comte2007penalized,sicard2021position}. An estimate of the diffusion parameter which improves upon the Kramer-Moyal estimate is given in \cite{ragwitz2001indispensable}. Further improvements have been made, such as estimates of the diffusion that are unbiased in the presence of measurement noise have also been developed \cite{vestergaard2014optimal} with the error quantified \cite{frishman2020learning}, and methods which allow for the estimation of underdamped Langevin equations \cite{bruckner2020inferring}.

The estimation of the drift and diffusion functions using these Kramer-Moyal estimates extends naturally into the SINDy framework. Stochastic Force Inference, as presented in \cite{frishman2020learning}, is a similar non-parametric identification method for SDEs, which differs in that it does not use a sparse solver. In \cite{boninsegna2018sparse}, the SINDy algorithm was used to estimate the parameters an SDE using these Kramer-Moyal estimates. This method was expanded on in \cite{dai2020detecting}; solution methods based on binning and cross validation were introduced to reduce the effects of noise. Callaham et. al \cite{callaham2021nonlinear} expand upon this method by adapting it to applications for which the random forcing cannot be considered white noise.

In the paper, we conduct a numerical analysis for using SINDy for stochastic system and introduce improved methods which give higher order convergence. As mentioned, in \cite{boninsegna2018sparse} the drift and diffusion are approximated using the Kramer-Moyal formulas. We demonstrate the convergence rates of the algorithm with respect to the sampling period and the length of the trajectory. The approximations given in \cite{boninsegna2018sparse} only give first order convergence with respect to the sampling frequency. A similar analysis of the Kramer-Moyal estimates based on binning can be found in \cite{chen2022non}. Additionally, since they only converge in expectation, we may require a long trajectory for the variance of the estimate to be tolerable. Combined, these can make the data requirements to use SINDy for an SDE very demanding. To help remedy this, we demonstrate how we can develop higher order approximations of the drift and diffusion functions for use in SINDy.

The paper is organized as follows: First, we will review the SINDy algorithm and some concepts from SDEs which we will be using in this paper. We will then conduct a numerical analysis of the algorithms presented in \cite{boninsegna2018sparse}, using the Ito-Taylor expansion of the SDE. Next, we will present new, higher order methods and show the convergence rates of these methods. Finally, we will test all of these methods on several numerical examples to demonstrate how the new methods allow us to compute far more accurate approximations of the system for a given sampling frequency and trajectory length.

\section{Sparse Identification of Nonlinear Dynamics (SINDy)}
\subsection{Overview}
Consider a system governed by the ordinary differential equation
\begin{equation}
\label{eq:ODE}
\dot{x}=f(x),~~~~x\in \R^d.
\end{equation}
If the dynamics of the system, $f$, are unknown, we would like to be able to estimate the function $f$ using only data from the system. The SINDy algorithm \cite{brunton2016discovering} estimates $f$ by choosing a dictionary of functions, $\theta=[\theta_1,\theta_2,...,\theta_k]$ and assuming $f$ can be expressed (or approximated) as a linear combination of these functions. The $i^{th}$ component of $f$, $f_i$, can then be expressed as
\[f_i(x)=\sum_{j=1}^k \theta_j(x)\alpha_{i,j}=\theta(x)\alpha_i ,\]
where $\theta=\begin{bmatrix} \theta_1 & \hdots & \theta_k\end{bmatrix}$ is a row vector containing the dictionary functions and $\alpha^i = \begin{bmatrix} \alpha^i_1 & \hdots &\alpha^i_k\end{bmatrix}^T$ is the column vector of coefficients. Given data for $f(x_j)$ and $\theta(x_j)$ for $j=1,...,n$, we can find the coefficients $\alpha_i$ by solving the minimization
\begin{equation}
\label{eq:leastsq}
\alpha_i=\underset{v}{argmin}\sum_{j=1}^n |f_i(x_j)-\theta(x_j)v|^2.
\end{equation}
This optimization can be solved by letting
\[\Theta=\begin{bmatrix}
\theta(x_1)\\\theta(x_2)\\\vdots\\\theta(x_n)
\end{bmatrix},~~~F=\begin{bmatrix}
f(x_1)\\f(x_2)\\\vdots\\f(x_n)
\end{bmatrix},~~~ \text{and} ~~~\alpha=\begin{bmatrix} \alpha^1 & \alpha^2 & \hdots & \alpha^d\end{bmatrix},\]
and computing $\alpha=\Theta^+F.$

\subsection{Approximating $f(x)$}
Typically, data for $f(x)$ cannot be measured directly. Instead, it is usually approximated using finite differences. The forward difference gives us a simple, first order approximation to $f$:
\begin{equation}
\label{eq:FD_det}
f(x(t))=\frac{x(t+\Delta t)-x(t)}{\Delta t}+O(\Delta t).
\end{equation}
Here $O(\Delta t)$ is the Landau ``big `O''' notation. The approximation (\ref{eq:FD_det}) is derived from the Taylor expansion of $x$, 
\begin{equation}
\label{eq:Taylor}
x(t+\Delta t)=x(t)+\dot{x}(t)\Delta t+\ddot{x}(t)\frac{\Delta t^2}{2}+...=x(t)+f(x(t))\Delta t+\frac{\partial f}{\partial x}\Big|_{x(t)}f(x(t))\frac{\Delta t^2}{2}+...,
\end{equation}
for $f$ sufficiently smooth. The Taylor expansion (\ref{eq:Taylor}) is also used to derive higher order methods, such as the central difference,
\begin{equation}
\label{eq:CD_det}
f(x)=\frac{x(t+\Delta t)-x(t-\Delta t)}{2\Delta t}+O(\Delta t^2).
\end{equation}
We can use these finite difference to populate the matrix $F$ used in the optimization (\ref{eq:leastsq}), knowing that we can control the error with a small enough step size.

\subsection{Sparse Solutions}
Since we are choosing an arbitrary dictionary of functions, $\{\theta_1,\hdots,\theta_k\}$, the conditioning of the minimization (\ref{eq:leastsq}) can become very poor. Additionally, if the the dictionary is large and contains many redundant functions, having a solution which contains only a few nonzero entries would help to provide a simple interpretable result. The SINDy algorithm addresses these by using a sparse solution to (\ref{eq:leastsq}). There are multiple methods for obtaining a sparse solution such as the least absolute shrinkage and selection operator (LASSO) or the sequentially thresholded least squares algorithm \cite{brunton2016discovering}. Using a sparse solution will give us a simpler identified system and improves the performance over the least squares solution.

\section{Review of SDEs}
Consider the Ito stochastic differential equation
\begin{equation}
\label{eq:Ito}
dX_t=\mu(X_t)dt+\sigma(X_t)dW_t
\end{equation}
where $X_t\in \R^d$ and $W_t$ is $d$-dimensional Brownian motion. The function $\mu:\R^d\to\R^d$ is the drift, a vector field which determines the average motion of system, while $\sigma:\R^d\to \R^{d\times d}$ is the diffusion function, which governs the stochastic forcing. The diffusion, $\sigma$, is also assumed to be positive definite. Motivated by SINDy, we wish to estimate $\mu$ and $\sigma^2$ from data. We note that we are estimating $\Sigma=\frac{1}{2}\sigma^2$ and not $\sigma$ directly. However, if $\sigma$ is positive definite, which is assumed, $\sigma^2$ uniquely determines $\sigma$.

\subsection{Ergodicity}
Since SINDy represents functions using the data vectors evaluated along the trajectory, we will need to relate the the data vectors to the functions represented in some function space. To do this, we will assume that the process $X_t$ has an ergodic measure $\rho$, so that both
\begin{equation}
\label{eq:erg}
\lim_{T\to\infty} \frac{1}{T}\int_0^T f(X_t)dt=\int_{\R^d} f(x)d\rho(x) ~~~~~\text{and}~~~~~\lim_{N\to\infty} \frac{1}{N} \sum_{i=0}^{N-1} f(X_{t_i})=\int_{\R^d} f(x)d\rho(x)
\end{equation}
hold almost surely. Some sufficient conditions that ensure that the SDE (\ref{eq:Ito}) generates a process with a stationary or an ergodic measure are given in e.g. \cite{khasminskii2011stochastic}.

With this ergodic measure, the natural function space to consider is the Hilbert space $L^2(\rho)$. For any two functions $f,g\in L^2(\rho)$, we can use time averages to evaluate inner products.
\begin{equation}
\label{eq:ergIP}
\lim_{T\to \infty}\frac{1}{T}\int_0^T g^*(X_t)f(X_t)dt=\lim_{N\to \infty}\frac{1}{N}\sum_{i=0}^{N-1}g^*(X_{t_n})f(X_{t_n})=\int_{\R^d}g^*f\,d\rho=\langle f,g\rangle .
\end{equation}
For notational simplicity, we will also use the brackets $\langle \cdot,\cdot \rangle$ to denote the matrix of inner products for two row vector-valued functions: if $f=\begin{bmatrix} f_1 & \hdots & f_k\end{bmatrix}$ and $g=\begin{bmatrix} g_1 & \hdots & g_l\end{bmatrix}$, 
\[\langle f,g\rangle^{i,j}=\langle f^j,g^i\rangle,~~~~~~~~\text{or equivelently,}~~~~~~~~ \langle f,g\rangle =\int_{\R^d}g^*f\,d\rho.\]

\subsection{Ito-Taylor Expansion}
In order to evaluate the performance of different SINDy methods on SDEs, we will need to use the Ito-Taylor expansion of the solution. Let $\Sigma=\frac{1}{2}\sigma^2$. Following the notation of \cite{kloeden1992stochastic}, let
\[L^0=\sum_{j=1}^d \mu^j\frac{\partial}{\partial x^j}+\sum_{j,l}^d (\Sigma)^{j,l}\frac{\partial^2}{\partial x^j\partial x^l}\]
be the operator for the Ito equation (\ref{eq:Ito}) and define the operators
\[L^j=\sum_{i=1}^d\sigma^{i,j} \frac{\partial}{\partial x^i}.\]
These operators will give us the coefficients for the Ito-Taylor expansion of a function $f$. Denoting $\Delta W^i_t=W^i_{t+\Delta t}-W^i_t$, the first couple of terms are
\begin{align*}
f(X_{t+\Delta t})=&f(X_t)+L^0f(X_t)\Delta t+\sum_{i=1}^dL^if(X_t)\Delta W^i_t+(L^0)^2f(X_t)\Delta t \,+\\
&\sum_{i=1}^d L^iL^0f(X_t)\int_t^{t+\Delta t}\int_t^{s_1}dW^i_{s_2}ds_1+\sum_{i=1}^d L^0L^if(X_t)\int_t^{t+\Delta t}\int_t^{s_1}ds_2dW^i_{s_1}+\hdots
\end{align*}
The general Ito-Taylor expansions can be found in Theorem 5.5.1 of \cite{kloeden1992stochastic}. We will use the Ito-Taylor expansion to develop estimates for $\mu^i$ and $\sigma^{i,j}$. For the purposes of this paper, we will be able to specialize to a few cases, which will allow us to quantify the error in our estimates while also being simpler to manipulate than the larger expansion.
\subsubsection{Weak Expansion}
The first specialization of the Ito-Taylor expansion will be a weak expansion, which will allow us to estimate the expected error in our estimate.
\begin{equation}
\label{eq:WeakExp}
\E(f(X_{t+\Delta t})|X_t)=f(X_t)+\sum_{m=1}^k (L^0)^mf(X_t) \frac{\Delta t^m}{m!}+R(X_t).
\end{equation}
with $R(X_t)=O(\Delta t^{m+1})$. 

This expansion follows from the Proposition 5.5.1 and Lemma 5.7.1 of \cite{kloeden1992stochastic}. Theorem 5.5.1 gives the general Ito-Taylor expansion, while Lemma 5.7.1 shows that all multiple Ito integrals which contain integration with respect to a component of the Weiner process have zero first moment. The remainder term is then a standard integral.

We will consider the expansion (\ref{eq:WeakExp}) with the functions $f(x)=x^i$ to get
\begin{equation}
\label{eq:WeakDrift}
\E(X^i_{t+\Delta t}|X_t)=X^i_t+\mu^i(X_t)\Delta t+\sum_{m=2}^k (L^0)^{m-1}\mu^i(X_t)\frac{\Delta t^m}{m!}+O(\Delta t^{k+1})
\end{equation}
to estimate the drift. To estimate the diffusion, we will let $f(x)=(x^i-X^i_t)(x^j-X^j_t)$, with $X_t$ held constant at the value at the beginning of the time step, to get
\begin{equation}
\label{eq:WeakDiff}
\E(f(X_{t+\Delta t})\,|\,X_t)=2\Sigma^{i,j}(X_t)\Delta t+g(X_t)\Delta t^2+O(\Delta t^3)
\end{equation}
where
\[g=L^0\Sigma^{i,j}+\mu^i\mu^j+\sum_{k=1}^d\left(\Sigma^{i,k}\frac{\partial \mu^j}{\partial x^k}+\Sigma^{j,k}\frac{\partial \mu^i}{\partial x^k}\right).\]

\subsubsection{Strong Expansions}
We will also use the strong Ito-Taylor expansion, which will give a bound on the variance of our estimates. These immediately follow from Proposition 5.9.1 of \cite{kloeden1992stochastic}. First, if we apply it to $f(x)=x^i$, we have 
\begin{equation}
\label{eq:StrongDrift}
X^i_{t+\Delta t}-X^i_t=\mu^i(X_t) \Delta t+\sum_{m=1}^d \sigma^{i,m}(X_t) \Delta W^m_t +R_t,
\end{equation}
where $\E(|R_t|^2|X_t)d\rho=O(\Delta t^{2})$.

Similarly, we can apply the same proposition to $f(x)=(x^i-X^i_t)(x^j-X^j_t)$, and which gives us (after moving around some of the terms)
\begin{equation}
\label{eq:StrongDiff}
(X^i_{t+\Delta t}-X^i_t)(X^j_{t+\Delta t}-X^j_t)=2\Sigma^{i,j}(X_t)\Delta t+\sum_{k,l=1}^d(\sigma^{k,i}\sigma^{l,j}(X_t)+\sigma^{k,j}\sigma^{l,i}(X_t))I_{(i,j)}+R_t,
\end{equation}
where $\E(|R_t|^2|X_t)=O(\Delta t^3)$ and $I_{(i,j)}=\int_0^{\Delta t}\int_0^{s_1}dW_{s_2}^idW_{s_1}^j$. When we create estimates of $\mu^i(X_t)$ and $\Sigma^{i,j}(X_t)$, the expansions (\ref{eq:StrongDrift}) and (\ref{eq:StrongDiff}) will be useful in bounding the variance of these two estimates.

\begin{remark}
	\label{re:Assumptions}
	For the expansions, it is implicit that we must assume that all (up to the necessary order) of the coefficient functions, $L^{a_1}L^{a_2}...L^{a_n}f$, satisfy the requirements with respect to the multiple Ito integrals set forth in chapter five of \cite{kloeden1992stochastic}. The conditions set forth are necessary for the Ito-Taylor expansions to be valid locally.
	
	Additionally, we will also need to assume that the remainder terms will be square integrable with respect to the ergodic measure. In particular, we will assume
	\[\int_{\R^d} |R(x)|^2d\rho(x)=O(\Delta t^{m+1})\]
	in the weak expansion and
	\[\int_{\R^d}R_2(x)^2 d\rho(x) = O(\Delta t)~~~~\left(\text{or}~~~~O(\Delta t^2)\right)\]
	in the strong expansions, where $R_2(x)=\E(|R_t|^2~|~X_t=x)$.
	This assumption will allow us to take time averages and expect them to be finite. Following the proofs in \cite{kloeden1992stochastic}, it can be seen that these can be guaranteed 
	by imposing similar integrability conditions on the coefficient functions with respect to the ergodic measure. This will often be the case, as the ergodic measure will decay rapidly towards infinity. A sufficiently strong condition to guarantee both the the integrability of the error is, for example, that both the diffusion and drift functions satisfy are smooth and the derivatives of all orders are bounded.
\end{remark}
\section{SINDy for Stochastic Systems}

Given data for the drift and diffusion matrix of (\ref{eq:Ito}), we can set up an optimization problem similar to (\ref{eq:leastsq}). Similar to the deterministic case, we can also approximate $\mu$ and $\Sigma$ using finite differences. As before, we assume we have a dictionary $\theta=[\theta_1,\theta_2,...,\theta_k]$ and that each of the components of $\mu$ and $\Sigma$ lie in the span of the components of $\theta$:
\begin{equation*}
\mu^i=\theta \alpha^i~~~~~~~\text{and}~~~~~~~\Sigma^{i,j}=\theta \beta^{i,j}.
\end{equation*}

Suppose we have the data from a trajectory of length $T$ with sampling period $\Delta t$. If we let $\Delta X^i_{t_n}=X^i_{t_{n+1}}-X^i_{t_n}$, we can approximate the drift using

\begin{equation}
\label{eq:FD}
\mu^i(X_{t_m})\approx\frac{X^i_{t_{m+1}}-X^i_{t_m}}{\Delta t}=\frac{\Delta X^i_t}{\Delta t}.
\end{equation}
Similarly, we can approximate the diffusion with
\begin{equation}
\label{eq:FDsquared}
\Sigma^{i,j}(X_{t_m})\approx\frac{(X^i_{t_{m+1}}-X^i_{t_m})(X^j_{t_{m+1}}-X^j_{t_m})}{2\Delta t}=\frac{\Delta X^i_{t_m}\Delta X^j_{t_m}}{2\Delta t}.
\end{equation}

It was shown in \cite{boninsegna2018sparse} that we can use the approximations (\ref{eq:FD}) and (\ref{eq:FDsquared}) to set up the minimization problems
\begin{equation}
\label{eq:MuLS}
\tilde{\alpha}^i=\underset{v}{argmin}\sum_{m=0}^{N-1} \left|\frac{\Delta X^i_{t_m}}{\Delta t}-\theta(X_{t_m})v\right|^2.
\end{equation}
and
\begin{equation}
\label{eq:SigLS}
\tilde{\beta}^{i,j}=\underset{v}{argmin}\sum_{m=0}^{N-1} \left|\frac{\Delta X^i_{t_m} \Delta X^j_{t_m}}{2\Delta t}-\theta(X_{t_m})v\right|^2.
\end{equation}
Under the assumptions set forth in Remark \ref{re:Assumptions}, we can show that as $\Delta t\to 0$ and $T\to \infty$, the coefficients given by (\ref{eq:MuLS}) and (\ref{eq:SigLS}) converge to the true coefficients; $\tilde{\alpha}^i\to \alpha^i$ and $\tilde{\beta}^{i,j}\to\beta^{i,j}$.

If we define the matrices
\begin{equation}
\Theta=\begin{bmatrix}
\theta(X_{t_0})\\\theta(X_{t_1})\\\vdots\\\theta(X_{t_{N-1}})
\end{bmatrix},~~~~~\text{and}~~~~~D^i=\begin{bmatrix}
\Delta X^i_{t_0}\\ \Delta X^i_{t_1}\\\vdots\\ \Delta X^i_{t_{N-1}}
\end{bmatrix},
\end{equation}
We can express (\ref{eq:MuLS}) and (\ref{eq:SigLS}) concisely as
\[\tilde{\alpha}^i=\underset{v}{argmin}\left\|\frac{D^i}{\Delta t}-\Theta v\right\|~~~~\text{and}~~~~\beta^{i,j}=\underset{v}{argmin}\left\|\frac{D^i\odot D^j}{2\Delta t}-\Theta v\right\|.\]
(Here $D^i\odot D^j$ represents the Hadamard, or element-wise, product.) These equations are solved by $\tilde{\alpha}_i=\Delta t^{-1}\Theta^+D^i$ and $\tilde{\beta}_{i,j}=(2\Delta t)^{-1}\Theta^+(D^i\odot D^j)$, respectively. 
\begin{theorem}
	\label{th:Convergence}
	Let $X_t$ be an ergodic drift-diffusion process generated by the SDE (\ref{eq:Ito}). Consider the optimization problems (\ref{eq:MuLS}) and (\ref{eq:SigLS}) using data from a trajectory of length $T$ sampled with frequency $\Delta t$. Suppose the components of $\theta$ are linearly independent and span the subspace $\mathcal{F}$, and that the assumptions on the Ito-Taylor expansions outlined in Remark \ref{re:Assumptions} are met. If $\mu^i$ or $\Sigma^{i,j}$ lie in $\mathcal{F}$, then the vectors given by corresponding optimization converges in probability to the true coefficients as $T\to \infty$ and $\Delta t \to 0$. That is, $\tilde{\alpha}^i\to\alpha^i$ or $\tilde{\beta}^{i,j}\to\beta^{i,j}$.
\end{theorem}

Theorem \ref{th:Convergence} was shown in \cite{boninsegna2018sparse} and will be implied by the stronger Theorems \ref{th:DriftOrd1} and \ref{th:DiffOrd1} which give rates of convergence. However, we will demonstrate the main reasoning behind the proof, as it will be informative to our later analysis. By the assumptions of Theorem \ref{th:Convergence} we have $\Theta$ has full rank and $\mu=\theta \alpha^i,$ $\Sigma^{i,j}=\theta \beta^{i,j}$.
\[\tilde{\alpha}^i=(\Theta^*\Theta)^{-1}\Theta^*\frac{D^i}{\Delta t}=\left(\frac{1}{N}\Theta^*\Theta\right)^{-1}\left(\frac{1}{N\Delta t}\Theta^*D^i\right),\]
where $N=T/\Delta t$ is the number of data samples. The first quantity can be evaluated using ergodicity, as $N\to \infty$
\[\frac{1}{N}\Theta^*\Theta=\frac{1}{N}\sum_{m=0}^{N-1}\theta^*(X_{t_m})\theta(X_{t_m})\xrightarrow{N}\langle \theta,\theta\rangle.\]
For the second expression, the definition of the stochastic integral gives us
\[\Theta^*D^i=\sum_{m=0}^{N-1}\theta^*(X_m)(X^i_{t_{m+1}}-X^i_{t_m})\xrightarrow{\Delta t} \int_{t_0}^{t_0+T} \theta^*dX^i\]
as $\Delta t\to 0$. Finally, using (\ref{eq:Ito}) and (\ref{eq:ergIP}), we can show
\begin{equation}
\label{eq:FakeConvergence}
\frac{1}{N\Delta t}\Theta^*D^i\xrightarrow{\Delta t}\frac{1}{T}\int_{t_0}^{t_0+T} \theta^*dX^i\xrightarrow{T}\langle \mu,\theta\rangle=\langle \theta,\theta\rangle\alpha^i
\end{equation}
as $\Delta t\to 0$ and $T\to \infty$. The limit as $\Delta t \to 0$ gives the convergence of the sum to the stochastic integral and the limit as $T\to \infty$ allows us to sample almost everywhere on the stationary measure for the ergodic convergence. Similarly, we can use the convergence
\[\sum_{m=0}^{N-1}\theta^*(X_{t_m})(X^i_{t_{m+1}}-X^i_{t_m})(X^j_{t_{m+1}}-X^j_{t_m})\xrightarrow{\Delta t} \int_{t_0}^{t_0+T}\theta^*d[X^i,X^j],~~~~\Delta t\to 0\]
to show that $\frac{1}{2N\Delta t}\Theta^*(D^i\odot D^j)\to \langle \Sigma^{i,j},\theta\rangle=\langle \theta,\theta\rangle\beta^{i,j}$. (Here $[X,Y]_t$ is the quadratic covariation process of $X_t$, and $Y_t$.) This would establish the result, except that we used the iterated limits $\Delta t \to 0$ and $T\to \infty$ in (\ref{eq:FakeConvergence}) without showing the double limit exists. This is where we would use the integrability assumptions in Remark \ref{re:Assumptions}, which are used in the proofs of Theorems \ref{th:DriftOrd1} and \ref{th:DiffOrd1}.

\bigskip

Theorem \ref{th:Convergence} demonstrates how the least squares solutions converge to the true coefficients of the SDE. However, the SINDy algorithm finds a sparse solution, which can greatly improve the accuracy of the results over the least squares solution. To set this up, the two optimizations (\ref{eq:MuLS}) and (\ref{eq:SigLS}) can be summarized using the normal equations,
\begin{equation}
\label{eq:MuNormal}
\Theta^*\Theta\tilde{\alpha}^i=\frac{1}{\Delta t}\Theta^*D^i\
\end{equation}
and
\begin{equation}
\label{eq:SigNormal}
\Theta^*\Theta\tilde{\beta}^{i,j}=\frac{1}{2\Delta t}\Theta^*(D^i\odot D^j).
\end{equation}
We can then solve equations (\ref{eq:MuNormal}) and (\ref{eq:SigNormal}) using a sparse solver, such as the one proposed in \cite{brunton2016discovering} to obtain a sparse solution.

\section{Numerical Analysis of Stochastic SINDy}
\label{sec:FirstOrder}
Theorem \ref{th:Convergence} claims that as $\Delta t\to 0$ and $T\to \infty$, the coefficients given by the (\ref{eq:MuLS}) and (\ref{eq:SigLS}) converge to the true parameters of the SDE (\ref{eq:Ito}) as $\Delta t\to 0$ and $T\to \infty$. However, for real experiments, there will be limits to the sampling frequency and the length of trajectory for which we can acquire data. In \cite{frishman2020learning}, the trajectory of the SDE was interpreted as a noisy transmission channel, and estimates on the relative squared errors were derived based on the information content of the signal.
 
In this section, we will use an alternate approach of deriving the error in the estimate based on the Ito-Taylor expansion of the SDE. We will use look at both the bias and variance of the approximations for finite $\Delta t$ and $T$. In particular, we derive the error with explicit constants (up to the leading order) in terms of the dictionary $\theta$ and the functions $\mu$ and $\sigma$.

In this setting, we will be using both ``big `O'" and ``little `o'" notation. The ``big `O'" notation will be used to denote convergence as $\Delta t\to 0$. These terms will come from the higher order error terms in the the estimators of $\mu^i(X_t)$ and $\Sigma^{i,j}(X_t)$. In particular, the constant in the ``big `O'" will depend only on the parameters of the SDE; it does not depend on the initial condition, trajectory length, or realization of the trajectory.
	
The ``little `o'" will denote convergence with respect to $T$. Specifically $o(1)$ denotes a function that goes to zero as $T\to\infty$. This will capture the ergodic convergence; the $o(1)$ term will be the error that comes from the finite trajectory failing to completely sample the ergodic measure.

The SINDy algorithm will give us vectors of coefficients, $\tilde{\alpha}^i$ and $\tilde{\beta}^{i,j}$, for the system. We will be interested in the error of these vectors relative to the true coefficients $\alpha^i$ and $\beta^{i,j}$,
\[err=\tilde{\alpha}^i-\alpha^i~~~~\text{or}~~~~err=\tilde{\beta}^{i,j}-\beta^{i,j}.\]
(We note that this error is specifically for the vector $\alpha^i$ or $\beta^{i,j}$ being estimated, even though it is not indexed. Since each vector is estimated separately, there should be no confusion.) This error will be a random variable depending on the realization of the system. To evaluate the performance of the algorithms, we will use the mean and variance of this error:
\[err_{mean}=\|\E(err)\|_2~~~~~\text{and}~~~~~err_{var}=Var(err)=\E(\|err-\E(err)\|_2^2).\]

The mean and variance of the error measure the bias and spread in the estimates $\tilde{\alpha}^i$ and $\tilde{\beta}^{i,j}$. These errors in the coefficients can be quantified using the errors in the estimates of $\mu^i$ and $\Sigma^{i,j}$ given in (\ref{eq:FD}) and (\ref{eq:FDsquared}) at each step. We will present the analysis for the drift coefficients, $\alpha^i$, noting that analysis for the diffusion follows the same path.

\subsection{Drift}
As mentioned, the error in $\tilde{\alpha}^i$ stems from the error in the approximation in (\ref{eq:FD})
\[\mu^i(X_{t_n})\approx\frac{X_{t_{n+1}}-X_{t_n}}{\Delta t}.\]
We can define the error
\[e_{t_n}=\frac{X^i_{t_{n+1}}-X^i_{t_n}}{\Delta t}-\mu^i(X_{t_n}).\]
The order of the error, $e_t$, at each time step will directly determine the error in the coefficients $\tilde{\alpha^i}$. We can use Ito-Taylor expansions for $X_t$ to bound both $\E(|e_t|)$ and $\E(|e_t|^2)$. The weak Ito-Taylor expansion (\ref{eq:WeakExp}) gives us
\begin{equation}
\label{eq:DriftErrm}
\E(e_t\,|\,X_t)=\frac{1}{\Delta t}\left(\mu^i(X_t)\Delta t+L^0\mu^i(X_t)\frac{\Delta t^2}{2}+O(\Delta t^3)\right)-\mu^i(X_t)=L^0\mu^i(X_t)\frac{\Delta t}{2}+O(\Delta t^2).
\end{equation}
Similarly, we can use the strong truncation (\ref{eq:StrongDrift}) to obtain
\[e_t=\sum_{m=1}^d \sigma^{i,m}(X_t)\frac{\Delta W^m_t}{\Delta t}+\frac{R_t}{\Delta t},\]
where $\E(|R_t|^2|X_t)=O(\Delta t^2)$. Then, taking the expectance of $e_t^2$, we get
\begin{equation}
\label{eq:DriftErrSq}
\E(|e_t|^2\,|\,X_t)=\sum_{m=1}^d \frac{\sigma^{i,m}(X_t)^2}{\Delta t}+O\left(\Delta t^{\frac{-1}{2}}\right).
\end{equation}
Now, let $E$ be the matrix containing the time samples of $e_t$, 
\[E=\begin{bmatrix}e_{t_0}&e_{t_1}&\hdots &e_{t_{N-1}}\end{bmatrix}^T=\frac{D^i}{\Delta t}-\Theta\alpha^i,\]
using $\theta(X_t)\alpha^i=\mu^i(X_t)$. Then we have
\begin{equation}
\label{eq:Err1}
err=\tilde{\alpha}^i-\alpha^i=\Theta^+\frac{D^i}{\Delta t}-\Theta^+\Theta\alpha=(\Theta^*\Theta)^{-1}\Theta^*E.
\end{equation}
Using ergodicity, we have
\begin{equation}
\label{eq:ThetaIP}
\left(\frac{1}{N}\Theta^*\Theta\right)^{-1}=\left(\langle \theta,\theta\rangle +o(1)\right)^{-1}=\langle\theta,\theta\rangle^{-1}+o(1),
\end{equation}
which allows us to evaluate the first term in (\ref{eq:Err1}):
\begin{equation}
\label{eq:Error}
err=(\langle \theta,\theta\rangle^{-1}+o(1))\left(\frac{1}{N}\Theta^*E\right).
\end{equation}
Bounding the mean and variance will follow from bounds on the mean and variance of $\frac{1}{N}\Theta^*E$.

\begin{theorem}
	\label{th:DriftOrd1}
	Consider the optimization problem given by (\ref{eq:FD}) and (\ref{eq:MuLS}). Then the bias is bounded by
	\[err_{mean}\leq \frac{C_1}{2}\left(\|L^0\mu^i\|_2+O(\Delta t)+o(1)\right)\Delta t\]
	and
	\[err_{var}\leq\frac{C_2}{T}\left(\sum_{m=1}^d\|\sigma^{i,m}\|_4^2+O\left(\Delta t^{\frac{1}{2}}\right)+o(1)\right),\]
	where 
	\begin{equation}
	\label{eq:Constants}
	C_1=\|\langle\theta,\theta\rangle^{-1}\|_2\|\theta\|_2~~~\text{and}~~~C_2=\|\langle\theta,\theta\rangle^{-1}\|_2^2\|\theta\|_4^2
	\end{equation}
	depend only on the choice of $\theta$.
\end{theorem}
As stated in Theorem \ref{th:DriftOrd1}, in expectation, the accuracy of our estimate depends primarily on the sampling period $\Delta t$, and not on the length of the trajectory. The length of the trajectory instead controls the variance of the estimate, which is proportional to $1/T$. Up to the leading term, the variance does not depend on the sampling period. These results previously appeared in \cite{frishman2020learning,bruckner2020inferring}, although our proof is different. This pattern will persist as we develop higher order methods for estimating the drift, where the sampling frequency determines the bias and the length of the trajectory determines the variance.

\begin{proof}
	For the mean error, we will need to bound the quantity $\frac{1}{N}\left\|\E\left(\Theta^*E\right)\right\|$. We have
	\[\E\left(\frac{1}{N}\Theta^*E\right)=\E\left(\frac{1}{N}\sum_{n=0}^{N-1}\theta^*(X_{t_n})e_{t_n}\right)=\E\left(\frac{1}{N}\sum_{n=0}^{N-1}\theta^*(X_{t_n})\E(e_{t_n}\,|\,X_{t_n})\right).\]
	Then, using ergodicity and (\ref{eq:DriftErrm}), we obtain
	\begin{align*}
	\E\left(\frac{1}{N}\Theta^*E\right)&=\E\left(\frac{1}{N}\sum_{n=0}^{N-1}\theta^*(X_{t_n})\left(\frac{\Delta t}{2}L^0\mu^i(X_{t_n})+O(\Delta t^2)\right)\right)
	\\&=\frac{\Delta t}{2}\left(\langle L^0\mu^i,\theta\rangle+o(1)\right) +O(\Delta t^2).
	\end{align*}
	Finally, using (\ref{eq:Error}), we get
	\begin{align*}
	\|\E(err)\|&=\left\|\left(\langle \theta,\theta\rangle^{-1}+o(1)\right)\right\|_2\left(\frac{\Delta t}{2}\left(\langle L^0\mu^i,\theta \rangle+o(1)\right)+O(\Delta t^2)\right)\\
	&\leq \|\langle\theta,\theta\rangle^{-1}\|_2\left(\|\theta\|_2\|L^0\mu^i\|_2+O(\Delta t)+o(1)\right)\frac{\Delta t}{2}=C_1\left(\|L^0\mu^i\|_2+O(\Delta t)+o(1)\right)\frac{\Delta t}{2}
	\end{align*}
	This bounds the mean error. To find the variance, we have
	\begin{align*}
	Var\left(\frac{1}{N}\Theta^*E\right)&\leq\E\left(\left\|\frac{1}{N}\Theta^*E\right\|_2^2\right)= \E\left(\left\|\sum_{n=0}^{N-1}\theta^*(X_{t_n})e_{t_n}\right\|_2^2\right)\leq\E\left(\sum_{n=0}^{N-1}\|\theta^*(X_{t_n})\|_2^2|e_{t_n}|^2\|\right)\\
	&=\E\left(\sum_{n=0}^{N_1}\|\theta(X_{t_n})\|_2^2\,\E\left(|e_{t_n}|^2\,|\,X_{t_n}\right)\right)
	\end{align*}
	Now, using (\ref{eq:DriftErrSq}) with this equation, we have
	\begin{align*}
	Var\left(\frac{1}{N}\Theta^*E\right)&\leq\E\left(\frac{1}{N^2}\sum_{n=0}^{N-1}\|\theta(X_{t_n})\|_2^2\left(\sum_{m=1}^d\frac{|\sigma^{i,m}|^2}{\Delta t}+O\left(\Delta t^{\frac{-1}{2}}\right)\right)\right)\\
	&=\frac{1}{N\Delta t}\left(\sum_{m=1}^d\langle (\sigma^{i,m})^2,\|\theta\|_2^2\rangle+O\left(\Delta t^{\frac{1}{2}}\right)+o(1)\right)\\
	&\leq\frac{1}{T}\|\theta\|_4^2\left(\sum_{m=1}^d\|\sigma^{i,m}\|_4^2+O\left(\Delta t^{\frac{1}{2}}\right)+o(1)\right).
	\end{align*}
	Then 
	\begin{align*}
	Var(err)&=\left(\|\langle \theta,\theta\rangle^{-1}\|_2^2+o(1)\right)\|\theta\|_4^2\left(\frac{1}{T}\left(\sum_{m=1}^d\|\sigma^{i,m}\|_4^2+O\left(\Delta t^{\frac{1}{2}}\right)+o(1)\right)\right)\\
	&=\frac{\|\langle\theta,\theta\rangle^{-1}\|_2^2\|\theta\|_4^2}{T}\left(\sum_{m=1}^d\|\sigma^{i,m}\|_4^2+O\left(\Delta t^{\frac{1}{2}}\right)+o(1)\right)
	\\&=\frac{C_2}{T}\left(\sum_{m=1}^d\|\sigma^{i,m}\|_4^2+O\left(\Delta t^{\frac{1}{2}}\right)+o(1)\right)
	\end{align*}
\end{proof}

\subsection{Diffusion}
The analysis of the diffusion coefficients follows the same argument. The approximation for $\Sigma^{i,j}$ given in (\ref{eq:FDsquared}) is
\[\Sigma^{i,j}(X_{t_m})\approx\frac{(X^i_{t_{m+1}}-X^i_{t_m})(X^j_{t_{m+1}}-X^j_{t_m})}{2\Delta t}=\frac{\Delta X^i_{t_m}\Delta X^j_{t_m}}{2\Delta t}.\]
Then we can define the error
\[e_t=\frac{(X^i_{t+\Delta t}-X^i_{t})(X^j_{t+\Delta t}-X^j_{t})}{2\Delta t}-\Sigma^{i,j}(X_t).\]
We can use the weak Ito-Taylor expansion (\ref{eq:WeakDiff}) to bound $\E(e_t\,|\,X_t)$:
\begin{equation}
\label{eq:DiffErrm}
\E(e_t\,|\,X_t)=g(X_t)\frac{\Delta t}{2}+O(\Delta t^2),~~~~~~~g=L^0\Sigma^{i,j}+\mu^i\mu^j+\sum_{k=1}^d \left(\Sigma^{i,k}\frac{\partial \mu^j}{\partial x^k}+\Sigma^{j,k}\frac{\partial \mu^i}{\partial x^k}\right).
\end{equation}
Similarly, the strong Ito-Taylor expansion (\ref{eq:StrongDiff}) gives us (see appendix \ref{App:ErrDiff1})
\begin{equation}
\label{eq:DiffErrsq}
\E(|e_t|^2\,|\,X_t)=\Sigma^{i,i}(X_t)\Sigma^{j,j}(X_t)+\Sigma^{i,j}(X_t)^2+O(\Delta t^{\frac{1}{2}}).
\end{equation}
\begin{theorem}
	\label{th:DiffOrd1}
	Consider the optimization problem given by (\ref{eq:FDsquared}) and (\ref{eq:SigLS}). Then the mean error is bounded by
	\[err_{mean}= \frac{C_1}{2}(\|g\|+O(\Delta t)+o(1))\Delta t,\]
	where
	\[g=L^0\Sigma^{i,j}+\mu^i\mu^j+\sum_{k=1}^d\left(\Sigma^{i,k}\frac{\partial \mu^j}{\partial x^k}+\Sigma^{j,k}\frac{\partial \mu^i}{\partial x^k}\right).\]
	The variance is bounded by
	\[err_{var}=\frac{C_2}{4}\left(\left\|\Sigma^{i,i}\Sigma^{j,j}+(\Sigma^{i,j})^2\right\|+O(\Delta t^{\frac{1}{2}})+o(1)\right)\frac{\Delta t}{T}.\]
	The constants $C_1$ and $C_2$ are the same as those given in (\ref{eq:Constants}).
\end{theorem}
\begin{proof}
	The proof follows that of Theorem \ref{th:DriftOrd1}, except using equations (\ref{eq:DiffErrm}) and (\ref{eq:DiffErrsq}) to bound $\left|\E(e_t\,|\,X_t)\right|$ and $\E\left(|e_t|^2\,|\,X_t\right)$, respectively.
\end{proof}
Similar to Theorem \ref{th:DriftOrd1}, the argument above shows that the mean error converges with order $\Delta t$. However, unlike the estimate for the drift, when estimating the diffusion the variance is proportional to both $\Delta t$ and $1/T$. Similar to the drift, these results were shown previously in \cite{frishman2020learning,bruckner2020inferring}. Later, we will see that the higher order estimates for the diffusion will also have variance proportional to $\Delta t/T$.

\section{Higher Order Methods}
\label{sec:HigherOrder}
	From Theorems \ref{th:DriftOrd1} and \ref{th:DiffOrd1} we can see that the quantities $\Delta t$, $T$, $C_1$, and $C_2$ will control the magnitude of the error. The constants, $C_1$ and $C_2$, depend only on the choice of the dictionary $\theta$, which determines the conditioning of the problem. The SINDy algorithm also uses a sparsity promoting algorithms which can improve the conditioning of the problem and force many of the coefficients to zero, which can reduce the error \cite{brunton2016discovering},\cite{boninsegna2018sparse}. However, even if the sparsity promoting algorithm chooses all of the correct coefficients, we have just shown that there is still a limit to the accuracy of the estimation determined by the sampling frequency and trajectory. The primary purpose of this section is to analyze alternate methods of approximating $\mu^i$ and $\Sigma^{i,j}$ which can improve the performance of SINDy (with respect to $\Delta t$).

	The methods above resulted from first order approximations (\ref{eq:FD}) and (\ref{eq:FDsquared}) of $\mu^i(X_t)$ and $\Sigma^{i,j}(X_t)$, respectively. Higher order approximations of these data points can in turn lead more accurate approximations of the functions in the output of SINDy. We can generate better approximations for the drift using multistep difference method. The use of linear multistep methods (LMMs) to estimate dynamics is investigated in \cite{keller2021discovery} for deterministic systems. While the estimates for the diffusion will be similar, they can not be achieved strictly using LMMs.
	
	In order to achieve a higher order approximation, we will need to use more data points in the approximation at each time step. As such, we will define
	\begin{equation} 
	\label{eq:MatrixDef2}
	\Theta_n=\begin{bmatrix}
	\theta(X_{t_n}) \\ \theta(X_{t_{n+1}}) \\ \vdots \\\theta(X_{t_{N+n-1}})
	\end{bmatrix}~~~~~\text{and}~~~~~D^i_n=\begin{bmatrix}
	X^i_{t_{n}}-X^i_{t_0} \\ X^i_{t_{n+1}}-X^i_{t_{1}} \\ \vdots \\ X^i_{t_{N+n-1}}-X^i_{t_{N-1}}
	\end{bmatrix}.
	\end{equation}
	With this definition, $\Theta_n$ contains the data of $\theta$ time delayed by $n$ steps. With the earlier definition of $\Theta,$ we have $\Theta=\Theta_0$. Similarly, $D^i_n$ contains the data for the change in $X$ over $n$ time steps, with $D^i_1=D^i$ using the earlier definition of $D^i$.
	
\subsection{Drift}
\label{sec:Drift}
First, we will look to make improvements on estimating the drift. These estimates will be simpler than those for the diffusion. As mentioned, these approximations are directly analogous to the linear multistep methods used in the simulation of deterministic systems.
\subsubsection{Second Order Forward difference}
The first order forward difference, which is used to approximate $\mu^i$ in Theorem \ref{th:DriftOrd1}, is also commonly used to approximate the derivative $f(x)$ in the differential equation $\dot{x}=f(x)$. In fact, if we compare the weak Ito-Taylor expansion (\ref{eq:WeakExp}) with the deterministic Taylor series for an ODE, (\ref{eq:Taylor}), we see that they are almost identical. There are many higher order methods which are used to approximate $f$ in the simulation of ODEs. By analogy, can expect that these methods would give an approximation  of the same order for $\mu^i$ (in expectation). One of the simplest of these would be the second order forward difference,
\begin{equation}
\label{eq:FDOrd2}
\mu^i(X_{t_n})\approx\frac{4(X_{t_{n+1}}-X_t)-(X_{t_{n+2}}-X_t)}{2\Delta t}=\frac{-3X^i_{t_n}+4X^i_{t_{n+1}}-X^i_{t_{n+2}}}{2\Delta t}.
\end{equation}
Similar to before we can define the error in this approximation to be
\[e_t=\frac{-3X^i_t+4X^i_{t+\Delta t}-X^i_{t+2\Delta t}}{2\Delta t}-\mu^i(X_t).\]
Using the weak Ito-Taylor expansion (\ref{eq:WeakExp}), it is easy to see that
\begin{equation}
\label{eq:DriftErrOrd2}
\E(e_{t_n}\,|\,X_{t_n})=-\frac{(L^0)^2\mu^i(X_{t_n})}{3}\Delta t^2+O(\Delta t^3),
\end{equation}
which shows that this method does indeed give a second order approximation of $\mu$. Using this approximation, we can set up a matrix formulation of (\ref{eq:FDOrd2}):
\[\Theta_0\alpha^i\approx\frac{1}{2\Delta t}\left(4D^i_1-D^i_2\right),\]
If we set up the normal equations, this becomes 
\begin{equation}
\label{eq:DriftOrd2Normal}
\Theta_0^*\Theta_0\tilde{\alpha}^i=\frac{1}{2\Delta t}\Theta_0^*\left(4D^i_1-D^i_2\right).
\end{equation}
\begin{theorem}
	\label{th:DriftOrd2}
	Consider the approximation $\tilde{\alpha}^i$ obtained from (\ref{eq:DriftOrd2Normal}). The mean error is bounded by
	\[\left\|\E(err)\right\|_2= \frac{C_1}{3}(\|(L^0)^2\mu^i\|+O(\Delta t)+o(1))\Delta t^2\]
	and the mean squared error by
	\[\E\left(\left\|(err)\right\|_2^2\right)=\frac{C_2}{T}\left(\sum_{j}^d\|\sigma^{i,j}\|_4^2+O(\Delta t^{\frac{1}{2}})+o(1)\right).\]
	The constants $C_1$ and $C_2$ are the same as those given in (\ref{eq:Constants}).
\end{theorem}
The proof of Theorem \ref{th:DriftOrd2} is similar to that of Theorem \ref{th:DriftOrd1}, but requires some extra algebraic manipulation, so it is included in appendix \ref{App:ErrorDrift1}.

\begin{remark}
	\label{Re:ItoVsStrat}
	These methods can easily be generalized to higher order methods using higher order finite differences, as will be done in section \ref{sec:GenDrift}. However, the least squares solution only yields correct results for forward differences. Other finite difference methods can cause certain sums to converge to the wrong stochastic integral. For example, a central difference approximation for $\mu^i$,
	\[\mu^i_{t}\approx\frac{X^i_{t+\Delta t}-X^i_{t-\Delta t}}{2\Delta t},\]
	gives us $\Theta_1 \alpha^i\approx \frac{1}{2\Delta t}D^i_2$. The normal equations for the least squares solution
	\begin{equation}
	\label{eq:CenDiff}
	\Theta_1^*\Theta_1 \tilde{\alpha}^i=\frac{1}{2\Delta t}\Theta_1^*D^i_2
	\end{equation}
	gives the wrong results, because as $\Delta t\to 0$, $\frac{1}{2}\Theta_1^*D^i_2$ converges to the Stratonovich integral instead of the Ito integral,
	\[\frac{1}{2}\Theta_1^*D^i_2\to \int_0^T\theta^*(X_t)\circ dX^i_t\neq \int_0^T \theta^*(X_t)\,dX^i_t,\]
	and $\tilde{\alpha}^i$ will not converge to the correct value. To prevent this, (\ref{eq:CenDiff}) can instead be solved using
	\[\Theta_0^*\Theta_1\tilde{\alpha}^i=\frac{1}{2\Delta t}\Theta_0^*D^i_2,\]
	which gives the proper convergence.	This amount to using $\Theta_0$ as a set of instrumental variables (see \cite{reiersol1945confluence}).
	
\end{remark}

\subsubsection{Trapezoidal Method}
The second order method above uses additional measurements of $X^i_t$ to provide a more accurate estimate of $\mu^i$. Alternatively, we can use multiple measurements of $\mu^i$ to better approximate the difference $X^i_{t+\Delta t}-X^i_t$. Consider the first order forward difference given by (\ref{eq:FD}).
\[\mu^i(X_{t_n})\approx\frac{X^i_{t_{n+1}}-X^i_{t_n}}{\Delta t}.\]
Theorem \ref{th:DriftOrd1} used this difference to give an order $\Delta t$ approximation of $\mu^i$. However, it turns out that $\frac{1}{2}(\mu^i(X_t)+\mu^i(X_{t+\Delta t}))$ gives a much better approximation of this difference:
\begin{equation}
\label{eq:TrapDifference}
\frac{1}{2}\left(\mu^i(X_{t_n})+\mu^i(X_{t_{n+1}})\right)\approx\frac{X^i_{t_{n+1}}-X^i_{t_n}}{\Delta t}.
\end{equation}
We will call this approximation the trapezoidal approximation, since this is exactly the trapezoidal method used in the numerical simulation of ODEs. If we consider the error in this equation,
\[e_t=\frac{X^i_{t_{n+1}}-X^i_{t_n}}{\Delta t}-\frac{1}{2}\left(\mu^i(X_{t_n})+\mu^i(X_{t_{n+1}})\right),\]
we can use the weak Ito-Taylor approximations of $X_t$ and $\mu^i(X_t)$  to show that
\begin{equation}
\label{eq:TrapErr}
\E(e_t\,|\,X_t)=-(L^0)^2\mu^i(X_t)\frac{\Delta t^2}{12}+O(\Delta t^3).
\end{equation}
This not only gives us a second order method, with respect to $\Delta t$, but the leading coefficient for the error is much smaller (by a factor of $1/8$) than the second order forward difference.

To set up the matrix formulation of (\ref{eq:TrapDifference}), we have
\begin{equation}
\label{eq:TrapMatrix}
\frac{1}{2}\left(\Theta_0+\Theta_1\right)\alpha^i\approx\frac{1}{\Delta t}D^i_1.
\end{equation}
We can multiply (\ref{eq:TrapMatrix}) by $\Theta_0^*$ on each side to obtain
\begin{equation}
\label{eq:TrapNormal}
\frac{1}{2}\Theta_0^*(\Theta_0+\Theta_1)\tilde{\alpha}^i=\frac{1}{\Delta t}\Theta_0^*D^i_1.
\end{equation}
We can use this equation analogously to the normal equation; we will solve for $\tilde{\alpha}^i$ either directly using matrix inversion or by using a sparse solver.
\begin{remark}
	\label{re:TrapSolve}
	We note that we cannot solve (\ref{eq:TrapMatrix}) using least squares,
	\[\tilde{\alpha}^i\neq \frac{2}{\Delta t}(\Theta_0+\Theta_1)^+D^i_1.\]
	Similar to Remark \ref{Re:ItoVsStrat}, this leads to sums converging to the wrong stochastic integral. In \cite{frishman2020learning}, a similar method was used which leverages the convergence to the Stratonovich integral to generate an approximation which better handles noise. Their method corrects for the Ito vs. Stratonovich differently than the one presented here, and requires an accurate estimate of the divergence of the diffusion function.
\end{remark}

\begin{theorem}
	\label{th:DriftTrap}
	Consider the estimation $\tilde{\alpha}^i$ given by solving (\ref{eq:TrapNormal}). The mean error is bounded by
	\[err_{mean}\leq C_1\frac{\Delta t^2}{12}(\|(L^0)^2\mu^i\|_2+O(\Delta t)+o(1))\]
	and
	\[err_{var}\leq \frac{C_2}{T}\left(\sum_{j=1}^d\|\sigma^{i,j}\|_2^2+O(\Delta t^{\frac{1}{2}})+o(1)\right).\]
\end{theorem}
\begin{proof}
	Letting $E$ be the matrix containing the samples of $e_t$. We have
	\[\frac{1}{\Delta t}D^i_1=\frac{1}{2}(\Theta_0+\Theta_1)\alpha^i+E.\]
	Using this in (\ref{eq:TrapNormal}) gives us
	\[\frac{1}{2}\Theta_0^*(\Theta_0+\Theta_1)\tilde{\alpha^i}=\frac{1}{2}\Theta_0^*(\Theta_0+\Theta_1)\alpha^i+\Theta_0^*E,\]
	so the error is
	\[err=\tilde{\alpha}^i-\alpha^i=\left(\frac{1}{2}\Theta_0^*(\Theta_0+\Theta_1)\right)^{-1}\Theta_0^*E.\]
	Since $\E(\theta(X_{t+\Delta t})|X_t)=\theta(X_t)+O(\Delta t)$, we can use ergodicity to evaluate
	\[\frac{1}{2N}\Theta_0^*(\Theta_0+\Theta_1)\to\langle \theta,\theta\rangle+O(\Delta t)+o(1).\]
	The proof of first inequality then follows the proof of Theorem \ref{th:DriftOrd1} and (\ref{eq:TrapErr}). The second inequality also follows using 
	\[\E\left(\|e_t\|_2^2~|~X_t=x\right)\leq \frac{1}{\Delta t}\sum_{m=1}^d |\sigma^{i,m}(x)|^2+O(\Delta t^{\frac{-1}{2}}),\]
	which can easily be derived using the Ito-Taylor expansions.	
\end{proof}

\subsubsection{General Method for Estimating Drift}
\label{sec:GenDrift}
	We have given methods which give second order estimates of $\alpha^i$. To generate methods which give even higher order approximations, we note the similarities of the above methods to linear multi-step methods used in the numerical simulation of ODEs. Using the general LMM as a guide, we set up a general method for approximating $\mu^i$:
	\begin{equation}
	\label{eq:GenMethod}
	\sum_{l=0}^k a_l\,\mu^i(X_{t_{n+l}})\approx\sum_{l=1}^p b_l\,(X^i_{t_{n+l}}-X^i_{t_{n}}),
	\end{equation}
	or
	\[\left(\sum_{l=0}^k a_l\Theta_l\right)\alpha^i\approx \sum_{l=1}^p b_l D^i_l.\]
	Keeping Remark \ref{Re:ItoVsStrat} in mind, we can solve this using
	\begin{equation}
	\label{eq:GenNormal}
	\left(\sum_{l=0}^{k}a_l\Theta_0^*\Theta_l\right)\tilde{\alpha^i}=b_l\sum_{l=1}^p\Theta_0^*D^i_l.
	\end{equation}
	The coefficients in (\ref{eq:GenMethod}) can be chosen to develop higher order methods. However, due to the stochastic nature of the problem, large amounts of data may be required to achieve the order in practice. We will need enough data to average over the randomness in the SDE, and the higher order methods can be sensitive to noise. More detailed investigation into the convergence of certain classes of methods for dynamics discovery can be found in \cite{keller2021discovery} for deterministic systems.

\subsection{Diffusion}
\label{sec:DiffOrd2}
In this section we will discuss improvements to the estimate for the diffusion. For some systems, particularly when the drift is large relative the diffusion, the first order approximation given above may not be sufficient to obtain an accurate estimate of the diffusion coefficient. Using similar ideas to the previous section we can use the Ito-Taylor expansions to develop more accurate estimates of $\Sigma^{i,j}(X_t)$. However, these methods will be more complex; in addition to samples of $X_t$, some of these methods may also require data from the drift, $\mu^i(X_t)$ and $\mu^j(X_t)$.
\subsubsection{Drift Subtraction}
Before discussing the higher order methods, we can make an improvement upon the first order method. In \cite{ragwitz2001indispensable}, Ragwitz and Kantz noted that by correcting for the effects of the drift in \ref{eq:FDsquared}, we can make significant improvements to the estimate. To derive their estimate, we use the Ito-Taylor expansion for $X_t$, which gives us
\[X^i_{t+\Delta t}-X^i_t=\mu(X_t)\Delta t+\sum_{m=1}^d\sigma(X_t)\Delta W^m_t+R_t,\]
where $\Delta W_t=W_{t+\Delta t}-W_t$ is the increment of a $d$-dimensional Wiener process and $R_t$ is the remainder term. This equation, with the remainder term excluded, actually gives the Euler-Marayama method for simulating SDEs. In essence, the approximation (\ref{eq:FDsquared}) uses 
\[X^i_{t+\Delta t}-X^i_t\approx\sum_{m=1}^d\sigma^{i,m}(X_t)\Delta W^m_t\]
to approximate the increment of the Wiener process. However, (\ref{eq:FDsquared}) tosses out the $\mu(X_t)\Delta t$ term because it is of a higher order. If we include it, we get the more accurate
\begin{equation}
\label{eq:WeinerEst}
\sum_{m=1}^d \sigma^{i,m} \Delta W^m_t=(X^i_{t+\Delta t}-X^i_t)-\mu(X_t)\Delta t-R_t.
\end{equation}
We can use this to generated a better approximation of $\Sigma^{i,j}$, 
\begin{equation}
\label{eq:MuSub}
\Sigma^{i,j}(X_t)\approx \frac{(X^i_{t+\Delta t}-X^i_t-\mu^i(X_t)\Delta t)(X^j_{t+\Delta t}-X^j_t-\mu^j(X_t)\Delta t)}{2\Delta t}.
\end{equation}
(We note that the estimate derived here is in slightly different form than that derived in \cite{ragwitz2001indispensable}, but will have a similar effect.), This approximation will be more accurate than (\ref{eq:FDsquared}), but it will have be of same order with respect to $\Delta t$. Letting $e_t$ be the error in (\ref{eq:MuSub}), we can use the weak Ito-Taylor expansion to show
\[\E(e_t\,|\,X_t)=f(X_t)\frac{\Delta t}{2}+O(\Delta t^2),~~~~~~~~~f=L^0\Sigma^{i,j}+\sum_{m=1}^d \left(\Sigma^{i,m}\frac{\partial \mu^j}{\partial x^m}+\Sigma^{j,m}\frac{\partial \mu^i}{\partial x^m}\right).\]
This gives an improvement over (\ref{eq:DiffErrm}) by removing the $\mu^i\mu^j$ term in $f$ (compared to Theorem \ref{th:DiffOrd1}). We note that this correction does not cancel all of the $O(\Delta t)$ terms in the error and thus does not improve the order of convergence. However, in systems where the drift dominates the diffusion the contributions of $\mu^i\mu^j$ will be large. For these systems, such as the Van-Der-Pol (\ref{eq:VanDerPol}) and Lorenz (\ref{eq:Lorenz}) examples presented in Section \ref{sec:Numerics}, the improvement will be large. In systems where the drift is typically small, such as the system with a double well potential (\ref{eq:DoubleWell}), the improvement will be modest.

In order to implement this method, we will need an approximation of $\mu^i$. However, we can use the methods above to represent the drift as $\mu^i(X_t)\approx\theta(X_t)\tilde{\alpha}^i$. We can use this to set up the matrix equations
\begin{equation}
\label{eq:MuSubNormal}
\Theta_0^*\Theta_0\tilde{\beta}^{i,j}=\frac{1}{\Delta t}(D^i_1-\Theta_0\tilde{\alpha}^i)\odot(D^j_1-\Theta_0\tilde{\alpha}^j),
\end{equation}
and solve for $\tilde{\beta}^{i,j}$.
\begin{remark}
	The equation (\ref{eq:MuSubNormal}) assumes that the same dictionary $\theta$ is used to estimate $\mu^i,\mu^j$ and $\Sigma^{i,j}$. In general, we could used separate dictionaries to estimate each of the parameters, since all we need are the approximations of the samples of $\mu^i(X_t)$ and $\mu^j(X_t)$ to estimate $\beta^{i,j}$.
\end{remark}

\subsection{Second Order Forward Difference}
While subtracting the drift from the differences \newline $X^i_{t+\Delta t}-X^i_t$ gives marked improvements, we can also generate a higher order method using a two step forward difference, similar to the drift. The analysis for the estimation of the diffusion constant using the two step forward difference is essentially identical to that of the drift, so we will go through it briefly. Define the approximation
\begin{equation}
\label{eq:Diff2Step}
\Sigma^{i,j}\approx\frac{4(X^i_{t+\Delta t}-X^i_t)(X^j_{t+\Delta t}-X^j_t)-(X^i_{t+2\Delta t}-X^i_{t})(X^j_{t+2\Delta t}-X^j_{t})}{4\Delta t}.
\end{equation}
As usual, letting $e_t$ be the error in this approximation, we can use the Ito-Taylor expansions (\ref{eq:WeakExp}) and (\ref{eq:StrongDiff}) to show that
\[\E(e_t)=O(\Delta t^2)~~~~~\text{and}~~~~~\E(|e_t|^2)=O(\Delta t).\]
This will gives us a second order method for the diffusion coefficients. We did not include the constants for the order $\Delta t^2$ for brevity, since the number of terms in the expressions can get quite large. We can use the approximation (\ref{eq:Diff2Step}) to set up the matrix equations
\begin{equation}
\label{eq:DiffOrd2Normal}
\Theta_0^*\Theta_0\tilde{\beta}^{i,j}=\frac{1}{4\Delta t}\Theta_0^*\left(4D^i_1\odot D^j_1-D^i_2\odot D^j_2\right),
\end{equation}
which we can solve for $\tilde{\beta}^{i,j}$.

\begin{theorem}
	\label{th:DiffOrd2}
	Consider the estimate $\tilde{\beta}^{i,j}$ given by solving (\ref{eq:DiffOrd2Normal}).
	Then we have
	\[err_{mean}=O(\Delta t^2)+o(1)\]
	and
	\[err_{var}=\frac{1}{T}O(\Delta t)+o(1/T).\]
\end{theorem}
The proof of Theorem \ref{th:DiffOrd2} is similar to the previous proofs. Additionally, we only give the leading order of the error, so deriving the bounds for $\E(e_t|X_t)$ and $\E(|e_t|^2|X_t)$ is simpler than the previous methods.

\subsubsection{Trapezoidal Method}
Extending the trapezoidal approximation to estimating the diffusion coefficient is slightly trickier. Let $\Delta X^i_t=X^i_{t+\Delta t}-X^i_{t}$. If we attempt use the analogue to (\ref{eq:TrapDifference}), we get
\[\Sigma^{i,j}(X_{t_{n+1}})+\Sigma^{i,j}(X_t)=\frac{\Delta X^i_{t_n}\Delta X^j_{t_n}}{\Delta t}+R_{t_n},\]
with
\[\E(R_{t_n})=f(X_{t_n})\Delta t+O(\Delta t^2),~~~~~~f=\mu^i\mu^j+\sum_{k=1}^d\left(\Sigma^{i,k}\frac{\partial \mu^j}{\partial x^k}+\Sigma^{j,k}\frac{\partial \mu^i}{\partial x^k}\right),\]
which is still only an order $\Delta t$ method. However, we already demonstrated in (\ref{eq:WeinerEst}) that correct the difference $\Delta X^i_t$ for the drift can improve our approximation of $\sum_{m=1}^d \sigma^{i,m}\Delta W^m_t$. We will use the same trick here, except we will improve upon (\ref{eq:WeinerEst}) by using the average values of $\mu^i$ and $\mu^j$ instead of the value at the left endpoint:
\[\sum_{m=1}^d \sigma^{i,m}\Delta W^m_t\approx (X_{t+\Delta t}-X_t)-\frac{\Delta t}{2}(\mu(X_t)+\mu(X_{t+\Delta t})).\]
If we use these differences to generate the trapezoidal method, we get 
\begin{equation}
\label{eq:TrapSig}
\Sigma^{i,j}(X_{t+\Delta t})+\Sigma^{i,j}(X_t)\approx\frac{\left(\Delta X^i_t-\frac{\Delta t}{2}(\mu^i(X_{t})+\mu^i(X_{t+\Delta t}))\right)\left(\Delta X^j_t-\frac{\Delta t}{2}(\mu^j(X_t)+\mu^j(X_{t+\Delta t}))\right)}{\Delta t}.
\end{equation}
If we consider the error in (\ref{eq:TrapSig}), using the appropriate Ito-Taylor expansions we can show (see appendix \ref{sec:AppDiffTrap})
\[|\E(e_t\,|\,X_t)|=O(\Delta t^2)~~~~~\text{and}~~~~~\E(|e_t|^2)=O(\Delta t).\]
Then, using the usual matrix notation, we can set up the equation
\begin{equation}
\label{eq:TrapSigNormal}
\Theta_0^*(\Theta_0+\Theta_1)\tilde{\beta}^{i,j}=\frac{1}{\Delta t}\left(D^i_1-\frac{\Delta t}{2} (\Theta_0+\Theta_1)\alpha^i\right)\odot\left(D^j_1-\frac{\Delta t}{2}(\Theta_0+\Theta_1)\alpha^j\right).
\end{equation}
We can solve this equation to get an order $\Delta t^2$ approximation of $\beta^{i,j}$.
\begin{theorem}
	\label{th:DiffTrap}
	Consider the estimate $\tilde{\beta}^{i,j}$ given by solving (\ref{eq:TrapSigNormal}).
	Then we have
	\[err_{mean}=O(\Delta t^2)+o(1)\]
	and
	\[err_{var}=\frac{1}{T}O(\Delta t)+o(1/T).\]
\end{theorem}
The proof of Theorem \ref{th:DiffTrap} is similar to the previous proofs, using the appropriate error bounds. Although the order of the error is identical to that of Theorem (\ref{th:DiffOrd2}), we will see that this method tends to have lower error. We did not include the constant terms for these errors for brevity, since the higher order Ito-Taylor expansions involve many terms.

\section{Numerical Examples}
\label{sec:Numerics}
\begin{table}[b!]
	\label{Tab:MethodSummary}
	\centering
	\def\arraystretch{1.5}
	\setlength{\tabcolsep}{15pt}
	\resizebox{\columnwidth}{!}{
		\begin{tabular}{|c||c|c||c|c|}
			\hline
			& \multicolumn{2}{l||}{Drift} & \multicolumn{2}{l|}{Diffusion} \\
			\hline \hline
			Name        & Equation      & Leading Error Term & Equation     &   Error   \\
			\hline
			FD-Ord 1        &        (\ref{eq:MuNormal})       &     $\frac{C_1}{2}\|L^0\mu^i\|_2\Delta t$	 
			&    (\ref{eq:SigNormal})          &       $O(\Delta t)$       \\
			\hline
			FD-Ord 2        &      (\ref{eq:DriftOrd2Normal})        &      $\frac{2C_1}{3}\|(L^0)^2\mu^i\|_2\Delta t^2$     &       (\ref{eq:DiffOrd2Normal})      &     $O(\Delta t^2)$         \\
			\hline
			Trapezoidal &   (\ref{eq:TrapNormal})          &    $\frac{C_1}{12}\|(L^0)^2\mu^i\|_2\Delta t^2$ &   (\ref{eq:TrapSigNormal})       &     $O(\Delta t^2)$         \\
			\hline
			Drift-Sub   &      -         &       -    &      (\ref{eq:MuSubNormal})        &      $O(\Delta t)$        \\
			\hline
		\end{tabular}
	}
	\caption{Summary of the methods for estimating the drift $(\mu^i)$ and the diffusion $(\Sigma^{i,j})$.}
\end{table}
In this section, we demonstrate the performance of the methods presented above on numerical examples. For each example, we will generate approximations $\tilde{\alpha}^i\approx \alpha^i$ and $\tilde{\beta}^{i,j}\approx \beta^{i,j}$. However, to present the data more simply, instead of computing the mean and mean squared error for each vector $\tilde{\alpha}^i$ and $\tilde{\beta}^{i,j}$, we will be aggregating the errors across all the coefficients. We will compute the mean error, normalized for the norms of $\alpha^i$ and $\beta^{i,j}$ using
\[Err_m=\left(\frac{\sum_{i=1}^d \|\E(\tilde{\alpha}^i)-\alpha^i\|_2^2}{\sum_{i=1}^d \|\alpha^i\|_2^2}\right)^{\frac{1}{2}} ~~~~\text{or}~~~~ Err_m=\left(\frac{\sum_{i\geq j\geq 1}^d \|\E(\tilde{\beta}^{i,j})-\beta^{i,j}\|_2^2}{\sum_{\substack{i\geq j \geq 1}}^d \|\beta^{i,j}\|_2^2}\right)^{\frac{1}{2}}.\]
Similarly, we will calculate the normalized variance
\[Err_{var}=\frac{\sum_{i=1}^d Var\left(\tilde{\alpha}^i\right)}{\sum_{i=1}^d \|\alpha^i\|_2^2}~~~~\text{or}~~~~ Err_{var}=\frac{\sum_{i\geq j\geq 1}^d Var\left(\tilde{\beta}^{i,j}\right)}{\sum_{\substack{i\geq j \geq 1}}^d \|\beta^{i,j}\|_2^2}.\]
Since these errors are based on aggregating the errors for all of the components of $\alpha^i$ or $\beta^{i,j}$, they will demonstrate the same convergence rates as in Theorems \ref{th:DriftOrd1}-\ref{th:DiffTrap}. The constants, however, may be different.

For each example, we will estimate the drift and diffusion using each of the methods described. The drift will be estimated using the first and second order forward differences, as well as the trapezoidal approximation. For the diffusion, we will use the first and second order forward differences, the drift-subtracted first order difference, and the trapezoidal method. For the drift-subtracted estimation, we will use the estimation for $\mu$ generated by the first order forward difference. Similarly, for the trapezoidal approximation for $\Sigma$, we will use the estimate generated by the trapezoidal approximation for $\mu$.

\subsection{Double Well Potential}	
Consider the SDE
\begin{equation}
\label{eq:DoubleWell}
dX_t=\left(-X_t^3+\frac{1}{2}X_t\right)\,dX_t+\left(1+\frac{1}{4}X_t^2\right)dW_t
\end{equation}
This equation represents a diffusion in the double well potential $U(x)=\frac{1}{4}x^4-\frac{1}{2}x^2$. This example is similar to one considered in \cite{boninsegna2018sparse}. Without the diffusion, the trajectories of this system will settle towards one of two fixed points, depending on which basin of attraction it started in. With the stochastic forcing, the trajectories will move around in one basin of attraction until it gets sufficiently perturbed to move to the other basin. We also note that for the majority of the trajectory, the state will be near the point where the drift is zero, so the dynamics will be dominated by the diffusion. At these points, the trajectory will behave similarly to Brownian motion.

\begin{figure}[b!]
	\label{fig:DoubleWellDrift}
	\begin{center}
		\includegraphics[scale=0.25]{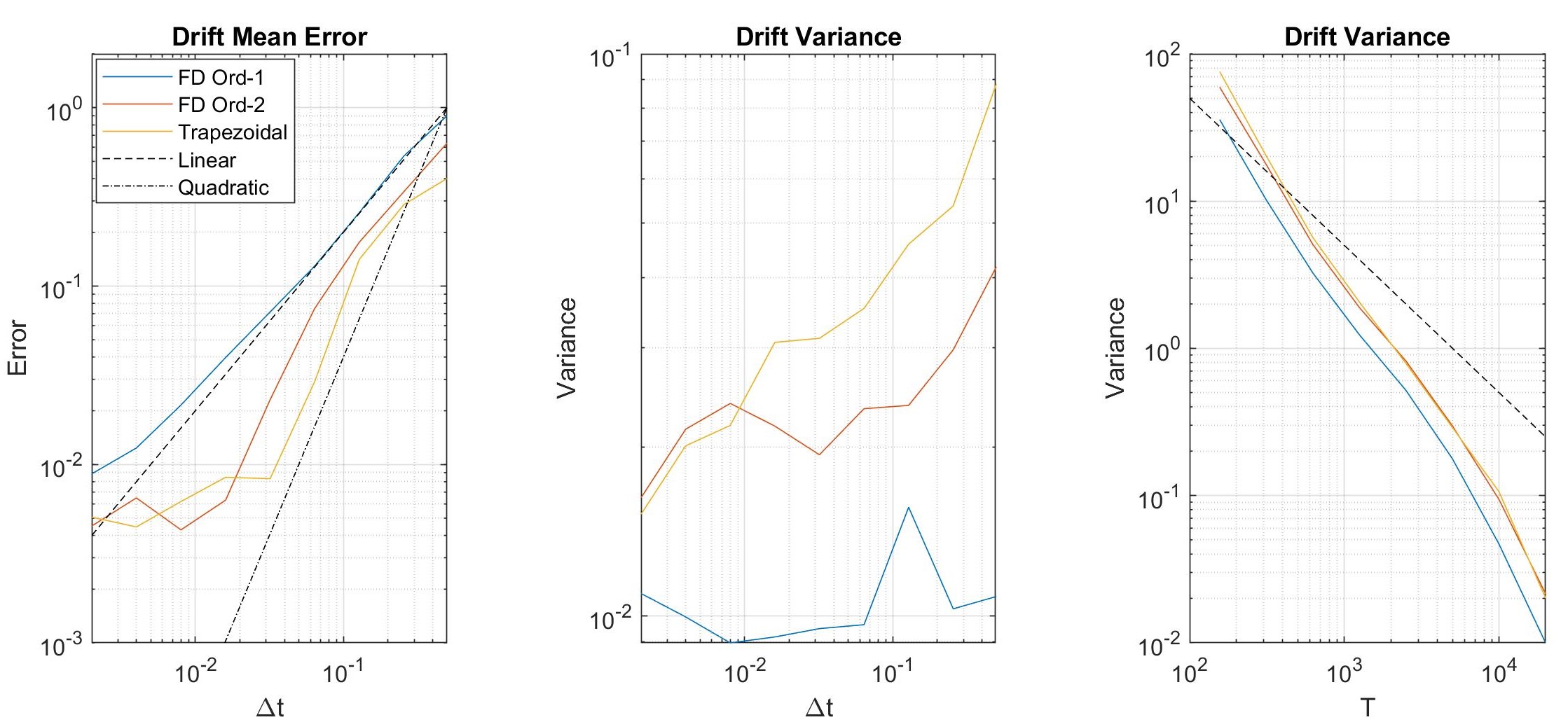}	
	\end{center}
	
	\caption{(Left) The mean error in the estimation of the drift coefficients for the double well system (\ref{eq:DoubleWell}) is plotted as a function of $\Delta t$. The error is approximated using 1,000 trajectories of length $T=20,000$. 
		\\(Center, Right) The variance for each method is plotted against the sampling period, $\Delta t$, and the trajectory length, $T$. The trajectory length is fixed at $T=20,000$ for the center plot, while the sampling period was fixed at $\Delta t=0.004=4\times 10^{-3}$ for the rightmost plot.}
\end{figure}

\begin{figure}[t!]
	\label{fig:DoubleWellDiff}
	\begin{center}
		\includegraphics[scale=0.25]{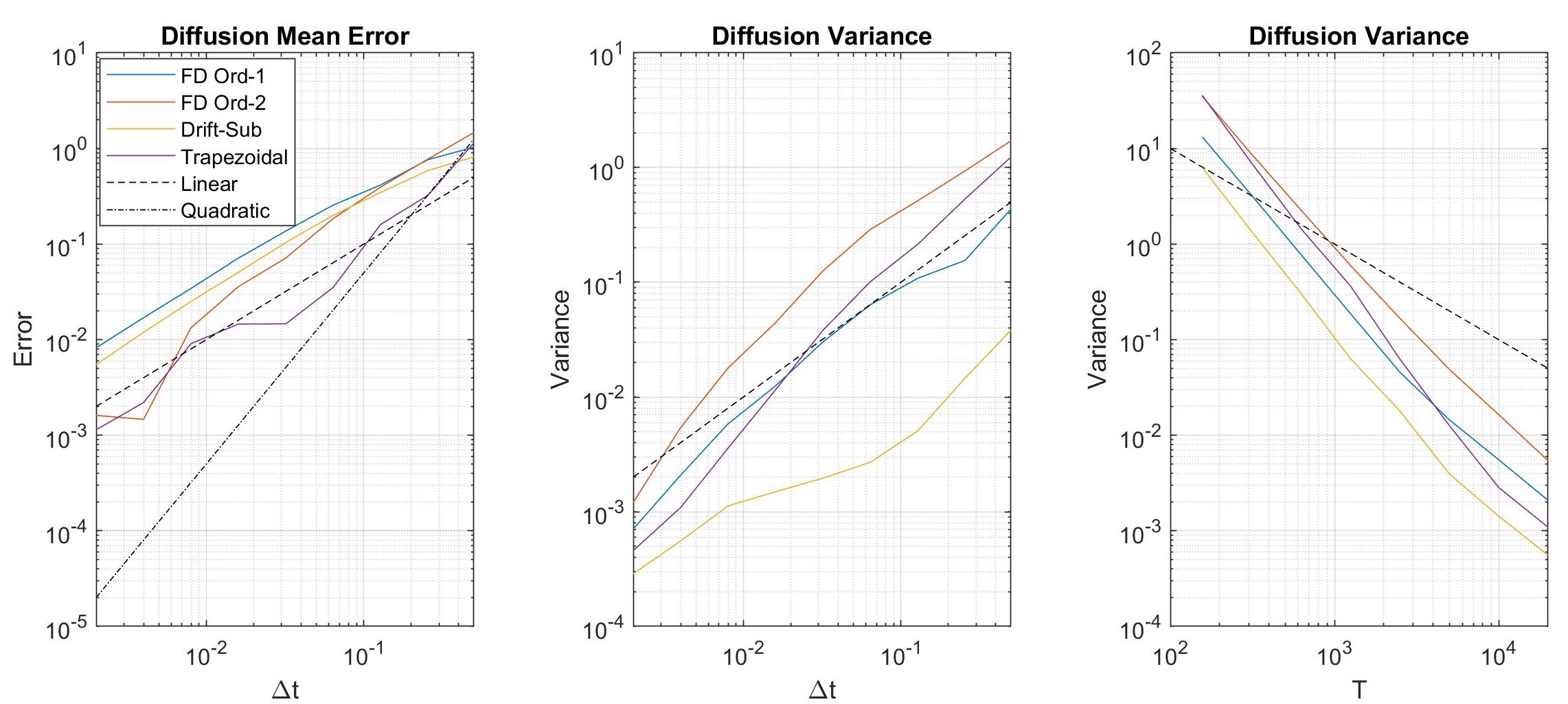}	
	\end{center}
	\caption{(Left) The mean error in the estimation of the diffusion coefficients for the double well system (\ref{eq:DoubleWell}) is plotted as a function of $\Delta t$. The error is approximated using 1,000 trajectories of length $T=20,000$. 
		\\(Center, Right) The variance for each method is plotted against the sampling period, $\Delta t$, and the trajectory length, $T$. The trajectory length is fixed at $T=20,000$ for the center plot, while the sampling period was fixed at $\Delta t=0.04=4\times 10^{-3}$ for the rightmost plot.}
\end{figure}

For the SINDy algorithm, we will use a dictionary of monomials in $x$ up to degree 14:
\[\theta(x)=\begin{bmatrix} 1 & x & \hdots & x^{14}\end{bmatrix}.\]
This basis will be used to estimate both the drift and diffusion. To generate the data for the algorithm, we simulated (\ref{eq:DoubleWell}) using the Euler-Maruyama method 1,000 times with a time step of $2\times 10^{-4}$ and a duration of 20,000. The initial condition was drawn randomly for each simulation from the standard normal distribution. The SINDy methods were then run on the data from each simulation for different sampling periods, $\Delta t$, and lengths of the trajectory, $T$. We use a minimum $\Delta t$ of $0.002$ so the simulation has a resolution of at least ten steps between each data sample. The truncation parameters for the sparse solver were set at $\lambda = 0.005$ for the drift and $\lambda=0.001$ for the diffusion.

As can be seen from from figure \ref{fig:DoubleWellDrift}, the expected errors in all three methods for the drift were converging to zero as $\Delta t\to 0$. For small $\Delta t$, the expected estimate was within $1\%$ of the true value. Additionally, the two higher order methods showed that, in expectation, they produce more accurate results and appear to converge more quickly, in line with Theorems \ref{th:DriftOrd1}, \ref{th:DriftOrd2}, and \ref{th:DriftTrap}. For these methods, the expected error was as much as an order of magnitude smaller, depending on the size of $\Delta t$. The convergence rate for the first order method scales linearly with $\Delta t$, while the higher order methods appear to scale quadratically until $\Delta t= 0.02,$ at which point there is likely not enough data overcome the variance in the estimate.

The variance, however, is rather large relative to the size of the expected error for all three methods. This is likely due to the system tending to settle towards the points $x=\pm 1/\sqrt{2}$ where the drift is zero. Near these points, the dynamics are dominated by the diffusion, making it difficult to estimate the drift. As can be seen (noting the scale of the center plot), the variance does not change a great amount as $\Delta t$ decreases, as is predicted for the estimates of the drift. As shown in the rightmost plot, the variance decreases as the length of the trajectory increases, slightly faster than linearly in $1/T$. In order to more fully benefit from using the higher order methods to the full extent, we would need a long enough trajectory to control the variance.

For the diffusion, figure \ref{fig:DoubleWellDiff} shows again that, as $\Delta t\to 0$, all of the methods do indeed converge in expectation. The Drift-Sub method slightly outperforms FD-Ord 1, the error is typically reduced by about $20\%-30\%$. Of the two higher order method, the trapezoidal method typically yields the best results, often an order of magnitude better than FD-Ord 1, although it does not appear to scale quadratically in $\Delta t$ as predicted by the theorem. This is likely due to a lack of sufficient data to average over the noise. FD-Ord 2 also gives substantial improvements for small $\Delta t$. Contrary to the drift, the variance in the estimate of the diffusion does decrease as $\Delta t$ goes to zero. The decrease appears to be proportional to $\Delta t$ and slightly faster than linear in $1/T$, which is roughly in line with the Theorems \ref{th:DiffOrd1}, \ref{th:DiffOrd2}, and \ref{th:DiffTrap}.

\subsection{Noisy Van-Der-Pol Oscillator}
Consider the ODE
\[\begin{bmatrix}
\dot{x}^1 \\ \dot{x}^2
\end{bmatrix}=\begin{bmatrix} x^2 \\ (1-(x^1)^2)x^2-x^1 \end{bmatrix}.\]
This is the Van-Der-Pol equation, which describes a nonlinear oscillator. We can perturb this equation by adding noise, we get the SDE
\begin{equation}
\label{eq:VanDerPol}
\begin{bmatrix}
dX^1_t \\ dX^2_t 
\end{bmatrix}=\begin{bmatrix} X^2_t \\ (1-(X^1_t)^2)X^2_t-X^1_t\end{bmatrix}dt+\sigma(X_t) dW_t,
\end{equation}
where $W_t$ is a two dimensional Wiener process. For the simulations, we let
\[\sigma(x)=\frac{1}{2}\begin{bmatrix} 1+0.3x^2 & 0 \\ 0 & 0.5+0.2 x^1\end{bmatrix}.\]
We chose this system to represent a different type of limiting behavior, and the estimation of a stochastic Van-Der-Pol oscillator was also considered in \cite{bruckner2020inferring}. For this system, the dynamics settle around a limit cycle. While they will have a certain amount of randomness, the trajectories will demonstrate an approximately cyclic behavior. In particular, this also means that the drift will rarely be near zero, as opposed to the previous example where the drift was often small.

\begin{figure}[b!]
	\label{fig:VanDerPolDrift}
	\begin{center}
		\includegraphics[scale=0.25]{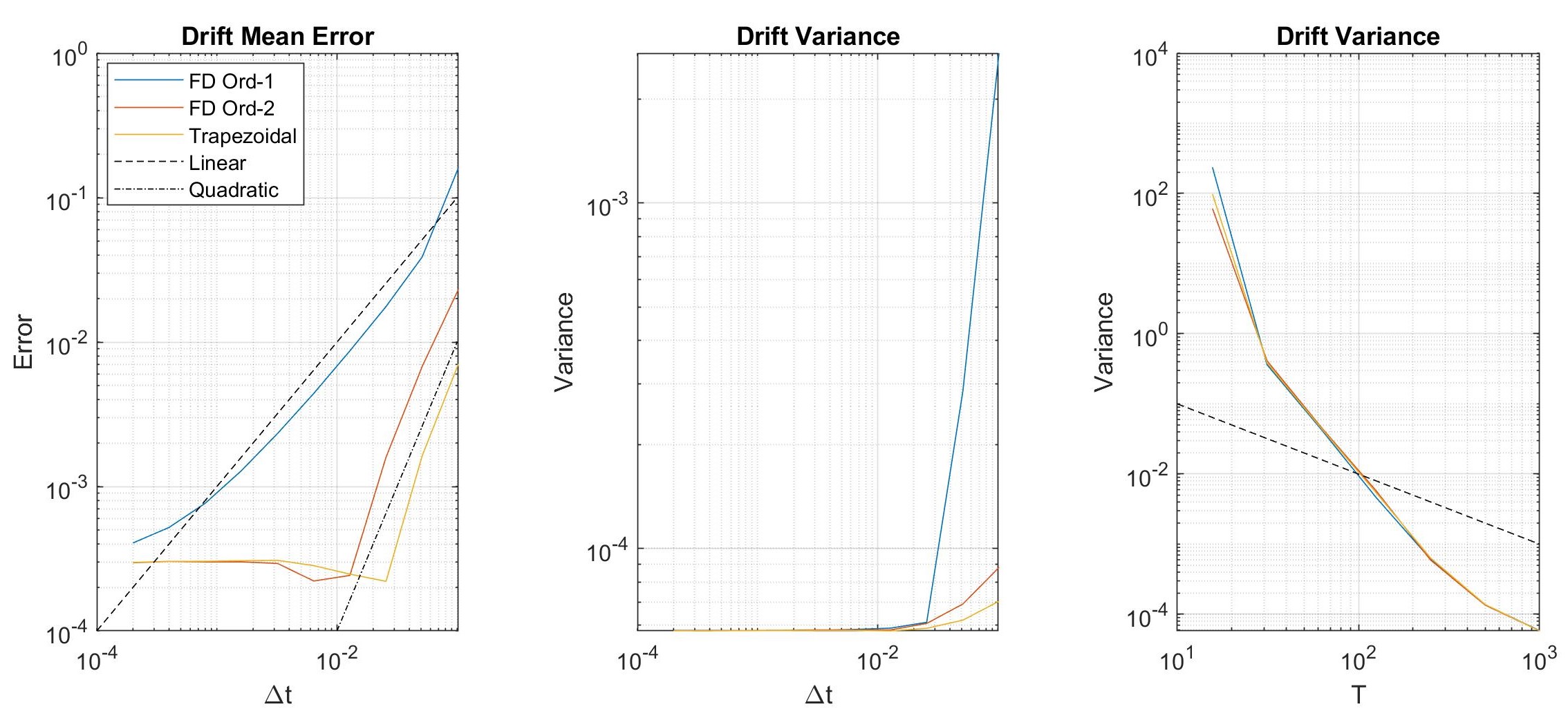}	
	\end{center}
	
	\caption{(Left) The mean error in the estimation of the drift coefficients for the Van-Der-Pol system (\ref{eq:VanDerPol}) is plotted as a function of $\Delta t$. The error is approximated using 1,000 trajectories of length $T=1,000$. 
		\\(Center, Right) The variance for each method is plotted against the sampling period, $\Delta t$, and the trajectory length, $T$. The trajectory length is fixed at $T=1,000$ for the center plot, while the sampling period was fixed at $\Delta t=0.008=8\times 10^{-3}$ for the rightmost plot.}
\end{figure}

\begin{figure}[t!]
	\label{fig:VanDerPolDiff}
	\begin{center}
		\includegraphics[scale=0.25]{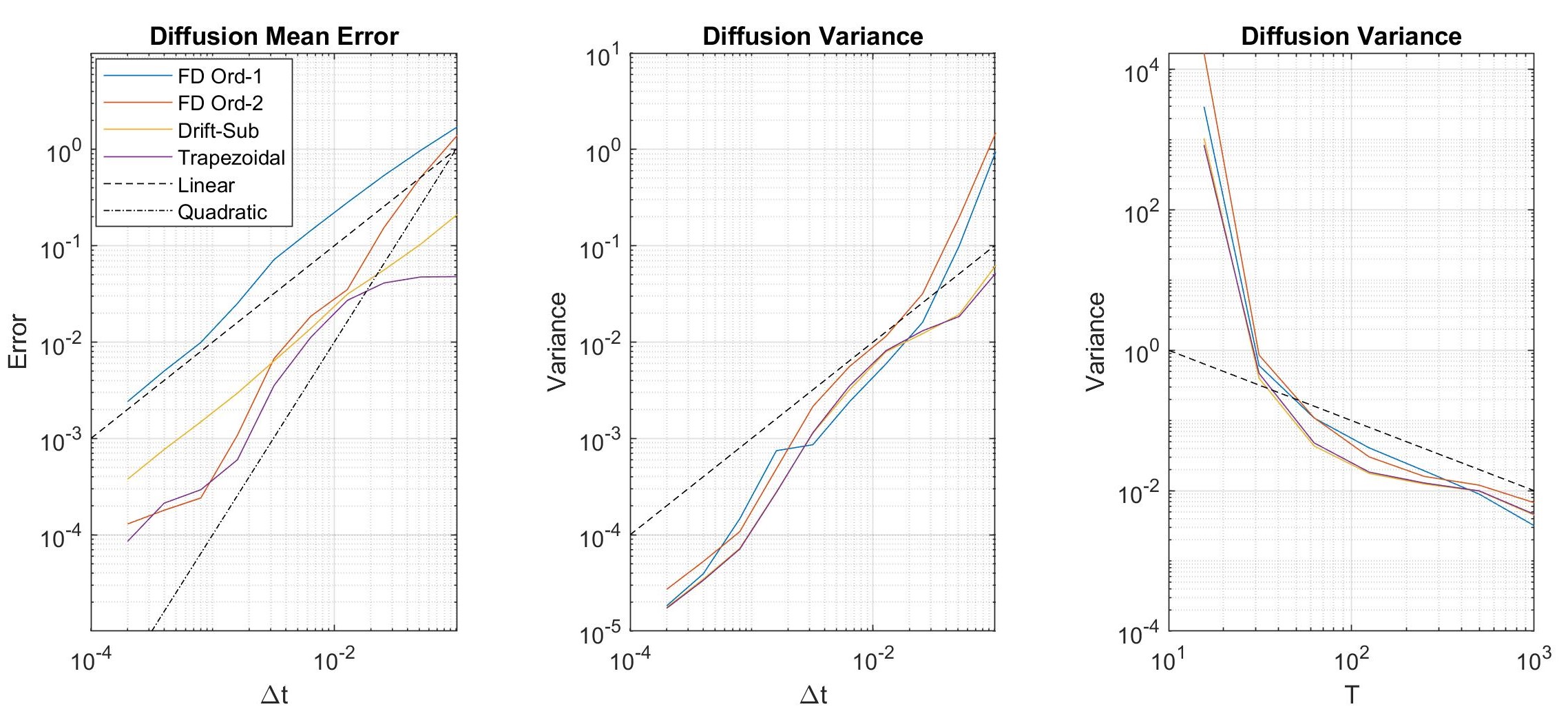}	
	\end{center}
	\caption{(Left) The mean error in the estimation of the diffusion coefficients for the Van-Der-Pol system (\ref{eq:VanDerPol}) is plotted as a function of $\Delta t$. The error is approximated using 1,000 trajectories of length $T=1,000$. 
		\\(Center, Right) The variance for each method is plotted against the sampling period, $\Delta t$, and the trajectory length, $T$. The trajectory length is fixed at $T=1,000$ for the center plot, while the sampling period was fixed at $\Delta t=0.008=8\times 10^{-3}$ for the rightmost plot.}
\end{figure}

The dictionary we will use for the SINDy algorithm consists of all monomials in $x^1$ and $x^2$ up to degree 6:
\[\theta(x)=\begin{bmatrix}
1 & x^1 & x^2 & x^1x^2 & \hdots & (x^1)^2(x^2)^4 & x^1(x^2)^5 & (x^2)^6
\end{bmatrix}.\]
This basis will be used to estimate both the drift and diffusion. To generate the data for the algorithm, we simulated (\ref{eq:VanDerPol}) using the Euler-Maruyama method 1,000 times with a time step of $2\times 10^{-5}$ and a duration of 1,000. Each component of the initial condition was drawn randomly for each simulation from the standard normal distribution. The SINDy methods were then run on the data from each simulation for different sampling periods, $\Delta t,$ and lengths of the trajectory, $T$. As before, we use $\Delta t\geq 2\times 10^{-4}$ to ensure that sampling period is at least 10 times the simulation time step. The truncation parameters for the sparse solver were set at $\lambda = 0.05$ for the drift and $\lambda=0.02$ for the diffusion.

In figure \ref{fig:VanDerPolDrift}, we first note that the variance very quickly drops to about $5\times 10^{-5}$ and stays roughly constant as $\Delta t$ decreases. This falls very much in line with the Theorems \ref{th:DriftOrd1}, \ref{th:DriftOrd2}, and \ref{th:DriftTrap} which assert that the variance does not depend on the sample frequency, it only decreases with the trajectory length $T$. For the expected error, the FD-Ord 2 and trapezoidal methods show drastic improvements over FD-Ord 1, with the trapezoidal method reducing the error by almost two orders of magnitude on some values of $\Delta t$. For the larger $\Delta t$, the slopes of the graphs demonstrate that these methods are converging at twice the order of the first order forward difference, as predicted by Theorems \ref{th:DriftOrd1}, \ref{th:DriftOrd2}, and \ref{th:DriftTrap}. However, both second order methods quickly reach a point where the performance remained constant at about $2\times 10^{-4}$. This is due to the lack of data to average over the random variation to sufficient precision. With sufficient data, we would expect the performance to continue to improve proportionally to $\Delta t^2$.

For the diffusion, figure \ref{fig:VanDerPolDiff} demonstrates a greater separation in the performance of the different methods compared to the double well system. Here, the FD-Ord 1 and drift subtracted methods both demonstrate the same first order convergence, as predicted in Theorem \ref{th:DiffOrd1}, but the drift subtracted method demonstrates a substantially lower error, ranging from half an order to almost a full order of magnitude better. FD-Ord 2 begins at roughly the same error as FD-Ord 1 for large $\Delta t$, but convergences faster, as predicted by Theorem \ref{th:DiffOrd2}, until it gives over an order of magnitude improvement for small $\Delta t$. Finally, although it is difficult to judge the speed of convergence for the trapezoidal method, it gives the most accurate results across all $\Delta t$. The variance for all of the methods behave similarly to the Double Well example and as expected, decreasing as $\Delta t\to 0$ and $T\to \infty$.

\subsection{Noisy Lorenz Attractor}
Consider the ODE
\[\dot{x}=\begin{bmatrix}
\dot{x}^1 \\ \dot{x}^2 \\ \dot{x}^3
\end{bmatrix}=\begin{bmatrix} 10(x^2-x^1) \\ x^1(28-x^3)-x^2 \\ x^1x^2-\frac{8}{3}x^3\end{bmatrix}=f(x).\]
This is the Lorenz system, which is famously a chaotic system exhibiting a strange attractor. If we perturb this equation by adding noise, we get the SDE
\begin{equation}
\label{eq:Lorenz}
dX_t=f(X_t)dt+\sigma(X_t) dW_t,
\end{equation}
where $W_t$ is a three dimensional Wiener process. The stochastic Lorenz process was also previous studied in the context of SDE identification in \cite{frishman2020learning}. For this example, we let
\[\sigma(x)=\begin{bmatrix} 1+\sin(x^2) & 0 & \sin(x^1)\\ 0 & 1+\sin(x^3)& 0 \\ \sin(x^1) & 0 & 1-\sin(x^2)\end{bmatrix}.\]

\begin{figure}[b!]
	\label{fig:LorenzDrift}
	\begin{center}
		\includegraphics[scale=0.25]{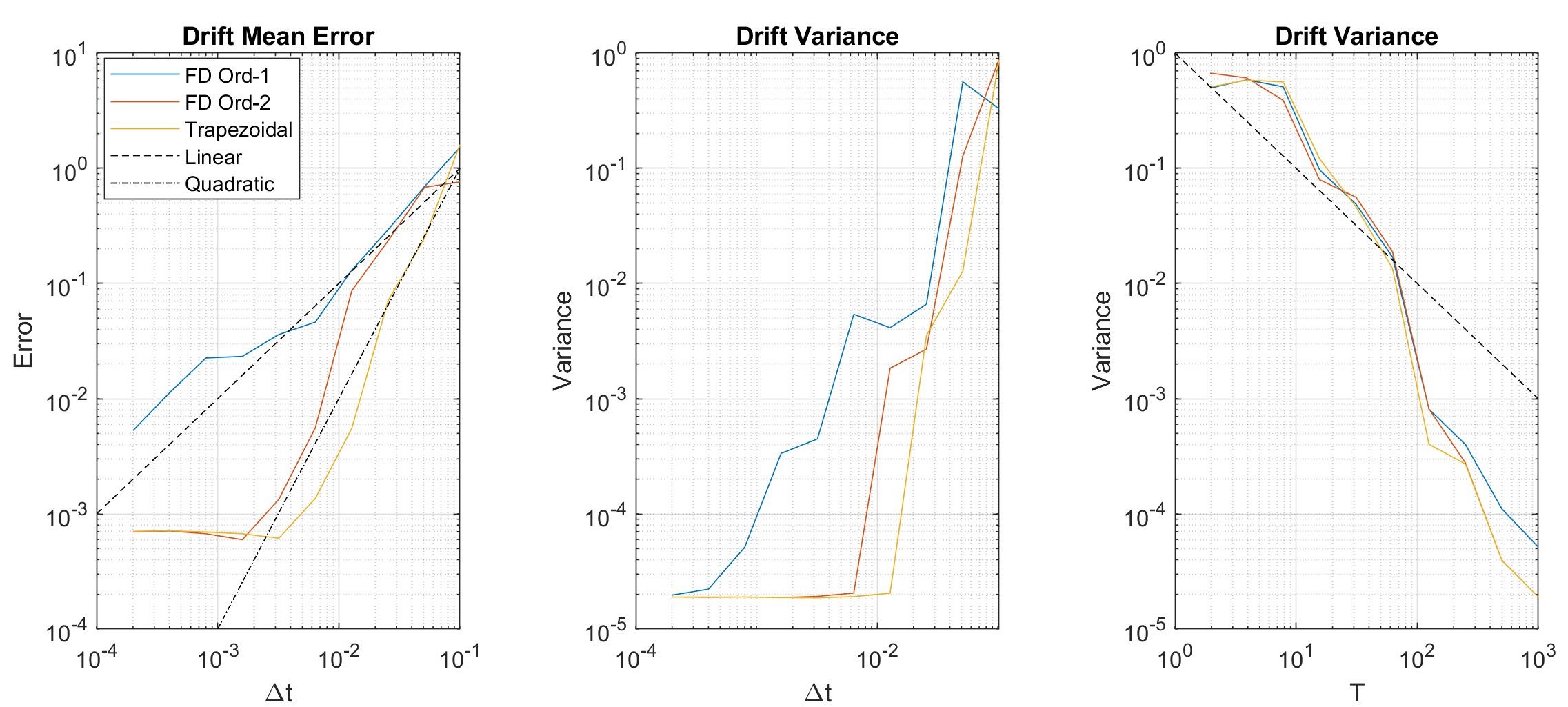}	
	\end{center}
	
	\caption{(Left) The mean error in the estimation of the drift coefficients for the Lorenz system (\ref{eq:Lorenz}) is plotted as a function of $\Delta t$. The error is approximated using 1,000 trajectories of length $T=1,000$. 
		\\(Center, Right) The variance for each method is plotted against the sampling period, $\Delta t$, and the trajectory length, $T$. The trajectory length is fixed at $T=1,000$ for the center plot, while the sampling period was fixed at $\Delta t=0.08=8\times 10^{-2}$ for the rightmost plot.}
\end{figure}

\begin{figure}[t!]
	\label{fig:LorenzDiff}
	\begin{center}
		\includegraphics[scale=0.25]{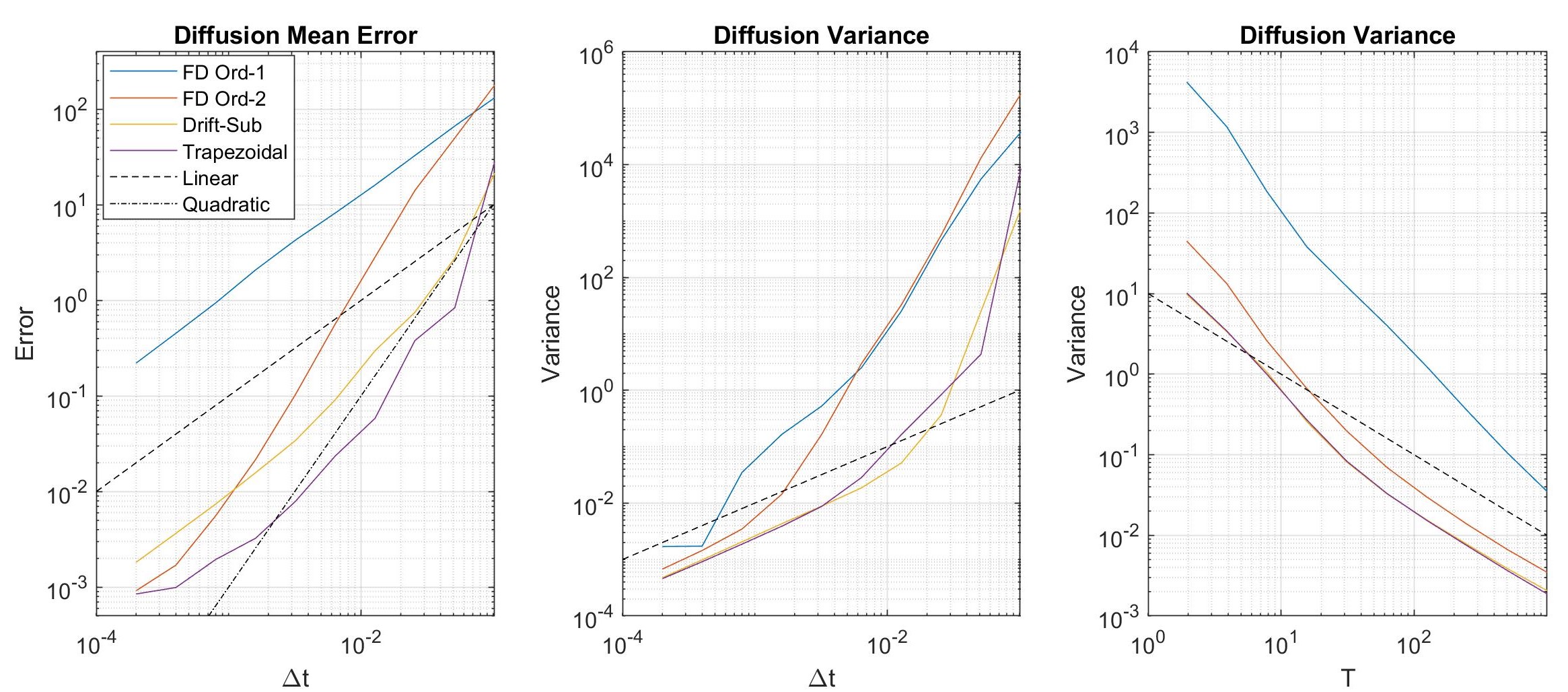}	
	\end{center}
	\caption{(Left) The mean error in the estimation of the diffusion coefficients for the Lorenz system (\ref{eq:Lorenz}) is plotted as a function of $\Delta t$. The error is approximated using 1,000 trajectories of length $T=1,000$. 
		\\(Center, Right) The variance for each method is plotted against the sampling period, $\Delta t$, and the trajectory length, $T$. The trajectory length is fixed at $T=1,000$ for the center plot, while the sampling period was fixed at $\Delta t=0.02=2\times 10^{-2}$ for the rightmost plot.}
\end{figure}

To generate the data for the algorithm, we simulated (\ref{eq:Lorenz}) using the Euler-Maruyama method 1,000 times with a time step of $2\times10^{-5}$ and a duration of 1,000. Each component of initial condition was drawn randomly for each simulation from the standard normal distribution. The SINDy methods were then run on the data from each simulation for different sampling periods, $\Delta t,$ and lengths of the trajectory, $T$. The truncation parameters for the sparse solver were set at $\lambda = 0.05$ for the drift and $\lambda=0.02$ for the diffusion.

We will use different dictionaries to estimate the drift and diffusion. For the drift, the dictionary consists of all monomials in $x^1,~x^2,$ and $x^3$ up to degree 4:
\[\theta(x)=\begin{bmatrix}
1 & x^1 & x^2 & \hdots & x^1x^2(x^3)^3 & (x^2)^2(x^3)^3 & x^2(x^3)^4 & (x^3)^5
\end{bmatrix}.\]
As before, figure \ref{fig:LorenzDrift} shows that the variance of the estimate for the drift decreases steadily as $T \to \infty$, while it approaches a minimum value as $\Delta t$ decreases and remains constant after reaching that minimum. In terms of the mean error, this example gives the clearest confirmation of the convergence rates demonstrated in Theorems \ref{th:DriftOrd1}, \ref{th:DriftOrd2}, and \ref{th:DriftTrap}. The slopes of the plots show that the error with FD-Ord 1 is roughly proportional to $\Delta t$, while the FD-Ord 2 and trapezoidal methods converge at double the rate. For small $\Delta t$, the second order methods do not seem to improve, due to the lack of sufficient data to compute the averages to high enough precision.

To estimate the diffusion, we used a dictionary consisting of all monomials in $\sin(x^1)$, $\sin(x^2)$, and $\sin(x^3)$ up to degree four:
\[\theta(x)=\begin{bmatrix}
1 & \sin(x^1) & \sin(x^2)& \hdots  & \sin(x^1)\sin(x^2)\sin^2(x^3) & \sin(x^2)\sin^3(x^3) & \sin^4(x^3)
\end{bmatrix}.\]
The error plot in figure \ref{fig:LorenzDiff} provides the most compelling example of the improvements of the higher order methods for estimating the diffusion. FD-Ord 1 clearly demonstrates its order one convergence  as $\Delta t\to 0$ (Theorem \ref{th:DiffOrd1}), but the error is quite large compared to the other methods. Even at our highest sampling frequency, $\Delta t=2\times 10^{-4}$, we only get slightly accurate results, with an error over 20\%. For this system, the drift subtracted method, although still first order, provides great improvements over FD-Ord 1, nearly two orders of magnitude better for most $\Delta t$. FD-Ord 2 also demonstrates the second order convergence given in Theorem \ref{th:DiffOrd2}, giving very accurate results for small $\Delta t$. Finally, the best performance again comes from the Trapezoidal method, which gives the best performance across all $\Delta t$. As expected from Theorem \ref{th:DiffTrap}, we can see that it converges faster than FD-Ord 1, but the convergence rate is not as clear as that of the other methods.

As for the variance, for all four methods it was roughly linear in $1/T$. It also decreased linearly with $\Delta_t$ once $\Delta t$ was small enough to give accurate estimates overall. However, the Trapezoidal and drift subtracted methods both showed a substantially lower variance for larger $\Delta t$. This is likely because the drift tends to dominate the diffusion in this system. Both the drift subtracted and trapezoidal methods correct for this, preventing the drift from having an effect on the estimate of the diffusion.

\section{Measurement Noise}
As demonstrated in the numerical examples above, the estimates for the drift and diffusion can yield accurate results provided the sampling frequency is high enough and the length of the trajectory is long enough. However, all of the numerical examples presented assumed ideal data. (i.e. there was no measurement noise in the observation of the state.) For real systems, this is rarely the case.

For the estimates presented above, the errors introduced be non-ideal data can be particular large for high sampling frequencies due to the measurement noise being divided by $\Delta t$. The effects of noise on diffusion estimates has been studied especially in the context of single particle tracking \cite{qian1991single,michalet2012optimal,vestergaard2014optimal}. Local estimates of the diffusion function which account for the noise have been presented (\cite{vestergaard2014optimal,frishman2020learning,bruckner2020inferring}), which can further be used to estimate the drift.

In this section, we will demonstrate how the methods presented above can be adapted to processes with measurement noise. For the drift, we will see that the approximations above can be directly adapted to handle noise using instrumental variables. For the diffusion, the approximation presented in \cite{vestergaard2014optimal} gives an unbiased estimate, and can be extended by methods similar to Section \ref{sec:DiffOrd2} for more accurate approximations.

For the duration of this section, we will assume that the measurement noise can be modeled as an i.i.d. Gaussian random vector with zero mean. Letting $Y_t$ be the noisy measurement and $\delta_t$ we have
\[Y_t=X_t+\delta_t,~~~~\delta_t~~\text{i.i.d.}\]
Further, we will also assume the noise is small enough that we can evaluate our dictionary functions accurately. More precisely for any dictionary function $\theta_k$, we will assume 
\[\theta_k(Y_t)=\theta(X_t+\delta_t)=\theta_k(X_t)+\sum_{i=1}^d (\nabla\theta_k)^T\delta_t+O(\|\delta_t\|^2)\]
and we can neglect the second order terms.

\subsection{Stochastic Force Inference}
In \cite{frishman2020learning}, the Stochastic Force Inference (SFI) methodology estimates the drift and diffusion functions of an SDE with both ideal and noisy data. When using noisy data, SFI first estimates the diffusion using the local estimate in \cite{vestergaard2014optimal}. Then, to measure the drift, SFI approximates a Stratonovich integral, which is unbiased with noise, and uses the estimate of the diffusion to correct the Stratonovich integral to the Ito one.

Letting $\Delta Y^i_t=Y^i_{t+\Delta t}-Y_t$, SFI approximates the diffusion using
\begin{equation}
\label{eq:SFIDiff}
\Sigma^{i,j}(X_{t})\approx \frac{(\Delta Y^i_{t}+\Delta Y^j_{t-\Delta t})(\Delta Y^j_t+\Delta Y^j_{t-\Delta t})}{4\Delta t}+\frac{\Delta Y^i_t\Delta Y^j_{t-\Delta t}+\Delta Y^i_{t-\Delta t}\Delta Y^j_t}{4\Delta t}.
\end{equation}
This approximation gives an estimate that converges with order $\Delta t$. Using the notation of (\ref{eq:MatrixDef2}) with the matrices populated using the noisy data $Y_t$, we can set up the normal equation to solve for $\tilde{\beta}^{i,j}$ as
\begin{equation}
\label{eq:SFIDiffNormal}
	\Theta_1^*\Theta_1\tilde{\beta^{i,j}}=\frac{1}{4\Delta t}\Theta_1^*\left[D^i_2\odot D^j_2+D^i_1\odot(D^j_2-D^j_1)+(D^i_2-D^i_1)\odot D^j_2\right].
\end{equation}
When estimating the drift, errors arise from the interaction of the noise in the approximation of $\mu$ and the effects of the noise on the dictionary function. To combat this, SFI evaluates a discretized Stratonovich integral and uses an estimate of the diffusion to convert the Stratonovich integral to an Ito one. The symmetry of the Stratonovich integral removes the bias introduced by the noise. The normal equations to summarize this method are
\begin{equation}
\label{eq:SFIDriftNormal}
\Theta_0^*\Theta_0\tilde{\alpha^i}=\frac{1}{2\Delta t}(\Theta_0+\Theta_1)^*D^i_1+\sum_{x=0}^{N-1}\sum_{j=1}^d  \Sigma^{i,j}(X_{t_n})\frac{\partial \theta}{\partial x^j}(X_{t_n}).
\end{equation}
For the evaluation $\Sigma^{i,j}(X_t)$ in this equation we can use the approximation (\ref{eq:SFIDiff}). The first term on the right hand side of this equation approximates the Stratonovich integral $\int_0^T \theta \circ dX_t$ while the second term corrects it to the Ito integral. For more detailed analysis of these methods, see \cite{frishman2020learning}. While this estimate is unbiased, it does have the disadvantages of requiring the differential of $\theta$ and using an estimate of $\Sigma$, which will also have some error.

\subsection{Instrumental Variables for Estimating Drift}
While the SFI method, it does have the disadvantages of using an estimate of $\Sigma$, which will also have some error, and requiring knowledge of the differential of $\theta$. However, we can adapt the estimates in section \ref{sec:Drift} to be unbiased. These methods will realize the same order of convergence (with respect to $\Delta t$) as the methods with ideal data in the large data limit. Additionally, they have the advantage of being simple to implement, and do not require the differential of $\theta$.

Consider the first order forward difference in the presence of noise
\[\mu(X_t)\approx \frac{Y_{t+\Delta t}-Y_t}{\Delta t}=\frac{X_{t+\Delta t}-X_t}{\Delta t}+\frac{\delta_{t+\Delta t}-\delta t}{\Delta t}.\]
Since this approximation is linear the noise $\delta_t$, its expected value is unaffected by the noise. All of the difference methods presented in section \ref{sec:Drift} have this property, since they are linear. The bias only comes from the interaction of the noise in the numerical derivatives and the noise in the dictionary.
\[\E\left(\theta_k(Y_t)\frac{\delta_{t+\Delta t}-\delta_t}{\Delta t}\right)\approx \E\left[\left(\theta_k(X)+(\nabla \theta_k)^T\delta_t\right)\left(\frac{\delta_{t+\Delta t}-\delta_t}{\Delta t}\right)\right]=Cov(\delta_t)\nabla\theta_k,\]
where $Cov(\delta_t)=\E(\delta_t\delta_t^T)$ is the covaraince matrix of $\delta_t$. However, if we use the previous dictionary values, $\theta(Y_{t-\Delta t})$, we have
\[\E\left(\theta(Y_{t-\Delta t})\frac{Y_{t+\Delta_t}-Y_t}{\Delta_t}\right)=\theta(X_{t-\Delta t})\frac{X_{t+\Delta t}-X_t}{\Delta t}\]
since the noise in $Y_{t-\Delta t},Y_t,$ and $Y_{t+\Delta t}$ are all independent. This amounts to using $\Theta(Y_{t-\Delta t})$ as a set of instrumental variables (see \cite{soderstrom1983instrumental}). The normal equation for this regression is
\begin{equation}
\label{eq:DriftOrd1IV}
\Theta_0^*\Theta_1\tilde{\alpha}^i=\frac{1}{\Delta t}\Theta_0 D^i_1.
\end{equation}
Similar to the first order method above, we can find a set of normal equations for the trapezoidal method for drift using instrumental variables.
\begin{equation}
\label{eq:DriftTrapIV}
\frac{1}{2}\Theta_0\left(\Theta_1+\Theta_2\right)\tilde{\alpha}^i=\frac{1}{\Delta t}\Theta_0D^i_1
\end{equation}

\subsection{Improving the Diffusion Estimate}
Equation (\ref{eq:SFIDiff}) gives an $O(\Delta t)$ approximation of $\Sigma^{i,j}(X_t)$ in expectation. We can improve this estimate in a similar manner to the Trapezoidal method for diffusion. Let
\[s^i(t)=Y^i_{t+\Delta t}-Y^i_t-\frac{\Delta t}{2}\left(\mu^i(Y_t)+\mu^i(Y_{t+\Delta t})\right),~~~~~q^i(t)=Y^i_{t+2\Delta t}-Y^i_t-\Delta t\left(\mu^i(Y_t)+\mu^i(Y_{t+2\Delta t})\right).\]
Then we can use the approximation
\begin{equation}
\label{eq:DiffTrapNoise}
\Sigma^{i,j}(X_t)+\Sigma^{i,j}(X_{t+2\Delta t})\approx\frac{1}{2\Delta t}\left(q^i(t)q^j(t)+s^i(t)s^i(t+1)+s^i(t+1)s^i(t)\right).
\end{equation}

Letting
\[S^i_n=\begin{bmatrix} s^i(t_n) & s^i(t_{n+1}) & \hdots s^i(t_{N+n-1})\end{bmatrix}^T,~~~\text{and}~~~ Q^i_n=\begin{bmatrix} q^i(t_n) & q^i(t_{n+1}) & \hdots q^i(t_{N+n-1})\end{bmatrix}^T,\]
we can set up the instrumental variables regression
\begin{equation}
\label{eq:TrapNoiseNormal}
\Theta_0^*(\Theta_1+\Theta_3)\tilde{\beta}^{i,j}=\frac{1}{2\Delta t}\Theta_0^*\left(Q^i_1\odot Q^j_1+S^i_1\odot S^j_2+S^i_2 \odot S^j_1\right)
\end{equation}
to solve for $\tilde{\beta}^{i,j}$.

\begin{figure}[b!]
	\label{fig:VanDerPolDriftNoise}
	\begin{center}
		\includegraphics[scale=0.25]{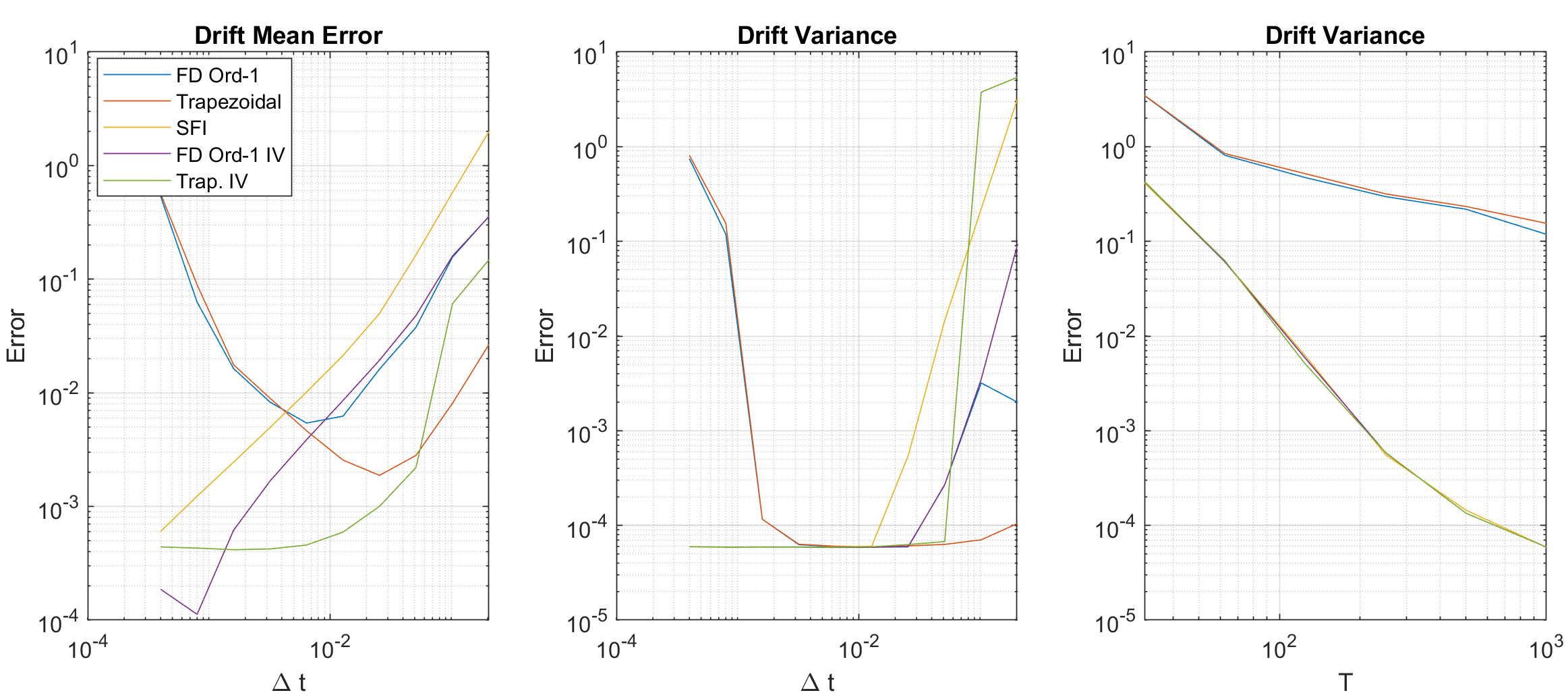}	
	\end{center}
	\caption{(Left) The mean error in the estimation of the drift coefficients for the Van-Der-Pol system (\ref{eq:VanDerPol}) with measurement noise is plotted as a function of $\Delta t$. The error is approximated using 1,000 trajectories of length $T=1,000$. 
		\\(Center, Right) The variance for each method is plotted against the sampling period, $\Delta t$, and the trajectory length, $T$. The trajectory length is fixed at $T=1,000$ for the center plot, while the sampling period was fixed at $\Delta t=0.008=8\times 10^{-3}$ for the rightmost plot.}
\end{figure}

\begin{figure}[t!]
	\label{fig:VanDerPolDiffNoise}
	\begin{center}
		\includegraphics[scale=0.25]{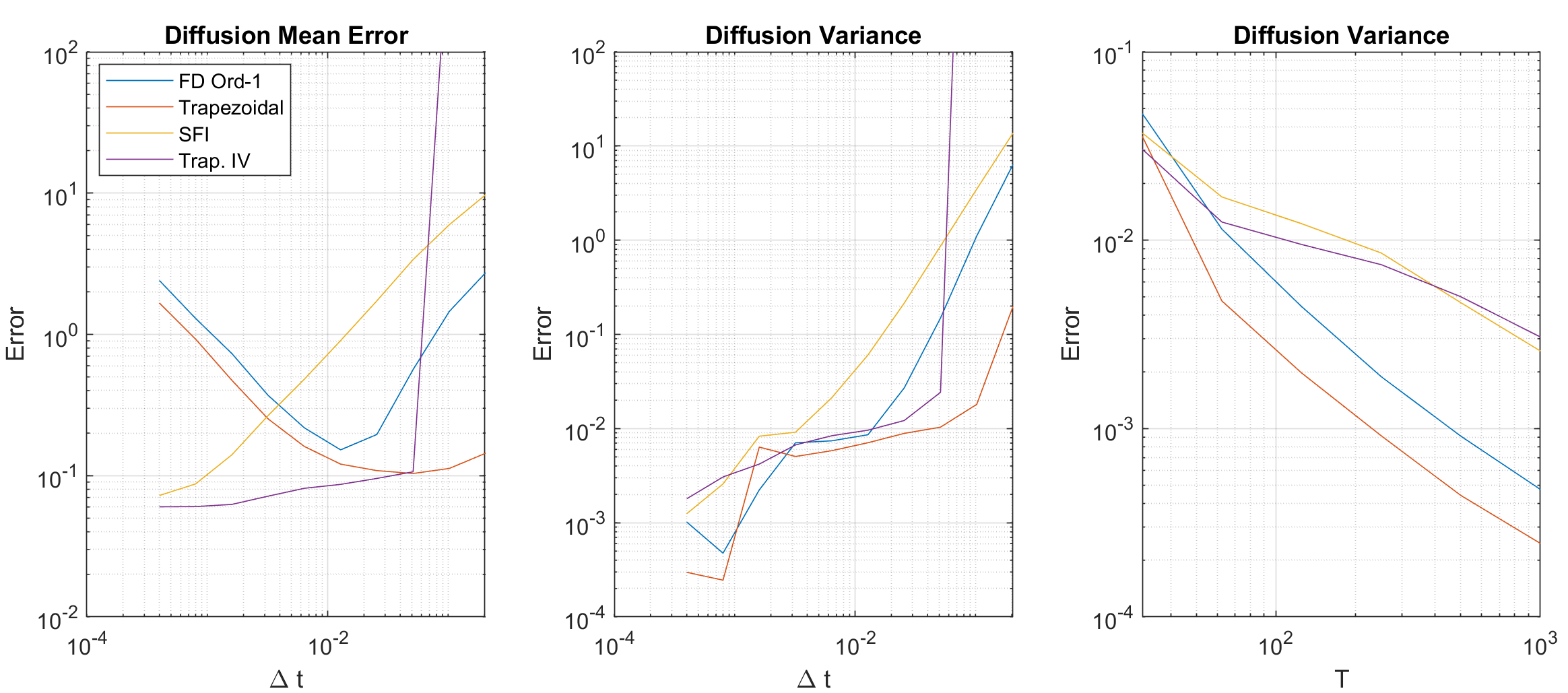}	
	\end{center}
	\caption{(Left) The mean error in the estimation of the diffusion coefficients for the Van-Der-Pol system (\ref{eq:VanDerPol}) with measurement noise is plotted as a function of $\Delta t$. The error is approximated using 1,000 trajectories of length $T=1,000$. 
		\\(Center, Right) The variance for each method is plotted against the sampling period, $\Delta t$, and the trajectory length, $T$. The trajectory length is fixed at $T=1,000$ for the center plot, while the sampling period was fixed at $\Delta t=0.008=8\times 10^{-3}$ for the rightmost plot.}
\end{figure}

\subsection{Van-Der-Pol Oscillator}
We now consider the stochastic Van-Der-Pol oscillator given by (\ref{eq:VanDerPol}) with measurement noise on the state $X_t$. The each component of the noise $\delta^i_t$ is drawn from a normal distribution with zero mean and a standard deviation of $0.02$. The system was simulated 1,000 times with a time step of $4\times 10^{-5}$ and a duration of 1,000. As before, we use $\Delta t\geq 4\times 10^{-4}$ to ensure that sampling period is at least 10 times the simulation time step. The truncation parameters for the sparse solver were set at $\lambda = 0.05$ for both the drift and the diffusion. Using this data, we test the noise-corrected methods presented in this section along with the first order and trapezoidal methods tested in section \ref{sec:Numerics}.

As can be seen from the error plots, while the methods which don't account for noise may somewhat accurate estimates for large $\Delta t$, as $\Delta t\to0$ they diverge from the true values of $\alpha^i$ and $\beta^{i,j}$. For the drift estimate, the SFI and methods using instrumental variables improve as $\Delta t\to 0$. The instrumental variables methods, however tend to give more accurate results, with the trapezoidal IV method greatly outperforming the others for larger $\Delta t$. For the diffusion, the estimate in SFI converges towards the true values of $\beta^{i,j}$ as $\Delta_t\to 0$. The trapezoidal method, however, reaches the same level of accuracy for much larger $\Delta t$.

\section{Conclusion}
As was shown in this and previous papers (\cite{boninsegna2018sparse},\cite{dai2020detecting},\cite{callaham2021nonlinear}), the SINDy algorithm can be used to accurately estimate the parameters of a stochastic differential equation. However, the significant amount of noise involved requires one to use either use great deal of data (i.e. a long time series) and/or methods which improve the robustness of SINDy to noise. Unfortunately, even if SINDy should identify all of the correct dictionary functions present in the dynamics, we showed that the sampling frequency limits the accuracy of the results when using the first order Kramer-Moyal formulas to estimate the drift and diffusion. The necessity for high sampling frequencies, combined with long trajectories, make SINDy a data hungry algorithm.

The higher order estimates presented in this paper allow us to overcome the $O(\Delta t)$ convergence given in \cite{boninsegna2018sparse,frishman2020learning}. With the higher order methods we can compute accurate estimations of the SDEs using far lower sampling frequencies. In addition to making SINDy a more accurate system identification tool, these improvements also greatly reduce the data requirements to feed the algorithm. By achieving accurate results at lower sampling frequencies we can reduce the data acquisition constraint, which makes SINDy a more feasible system identification method for SDEs.

\newpage
\bibliographystyle{plain}
\bibliography{Sources.bib}

\begin{thebibliography}{10}

\bibitem{boninsegna2018sparse}
Lorenzo Boninsegna, Feliks N{\"u}ske, and Cecilia Clementi.
\newblock Sparse learning of stochastic dynamical equations.
\newblock {\em The Journal of chemical physics}, 148(24):241723, 2018.

\bibitem{bruckner2020inferring}
David~B Br{\"u}ckner, Pierre Ronceray, and Chase~P Broedersz.
\newblock Inferring the dynamics of underdamped stochastic systems.
\newblock {\em Physical review letters}, 125(5):058103, 2020.

\bibitem{brunton2016discovering}
Steven~L Brunton, Joshua~L Proctor, and J~Nathan Kutz.
\newblock Discovering governing equations from data by sparse identification of
  nonlinear dynamical systems.
\newblock {\em Proceedings of the national academy of sciences},
  113(15):3932--3937, 2016.

\bibitem{brunton2016sparse}
Steven~L Brunton, Joshua~L Proctor, and J~Nathan Kutz.
\newblock Sparse identification of nonlinear dynamics with control (sindyc).
\newblock {\em IFAC-PapersOnLine}, 49(18):710--715, 2016.

\bibitem{callaham2021nonlinear}
Jared~L Callaham, J-C Loiseau, Georgios Rigas, and Steven~L Brunton.
\newblock Nonlinear stochastic modelling with langevin regression.
\newblock {\em Proceedings of the Royal Society A}, 477(2250):20210092, 2021.

\bibitem{chen2022non}
Xi~Chen and Ilya Timofeyev.
\newblock Non-parametric estimation of stochastic differential equations from
  stationary time-series.
\newblock {\em Journal of Statistical Physics}, 186:1--31, 2022.

\bibitem{comte2007penalized}
Fabienne Comte, Valentine Genon-Catalot, and Yves Rozenholc.
\newblock Penalized nonparametric mean square estimation of the coefficients of
  diffusion processes.
\newblock {\em Bernoulli : official journal of the Bernoulli Society for
  Mathematical Statistics and Probability}, 13(2):514--543, 2007.

\bibitem{dai2020detecting}
Min Dai, Ting Gao, Yubin Lu, Yayun Zheng, and Jinqiao Duan.
\newblock Detecting the maximum likelihood transition path from data of
  stochastic dynamical systems.
\newblock {\em Chaos: An Interdisciplinary Journal of Nonlinear Science},
  30(11):113124, 2020.

\bibitem{fasel2022ensemble}
Urban Fasel, J~Nathan Kutz, Bingni~W Brunton, and Steven~L Brunton.
\newblock Ensemble-sindy: Robust sparse model discovery in the low-data,
  high-noise limit, with active learning and control.
\newblock {\em Proceedings of the Royal Society A}, 478(2260):20210904, 2022.

\bibitem{friedrich2011approaching}
Rudolf Friedrich, Joachim Peinke, Muhammad Sahimi, and M~Reza~Rahimi Tabar.
\newblock Approaching complexity by stochastic methods: From biological systems
  to turbulence.
\newblock {\em Physics Reports}, 506(5):87--162, 2011.

\bibitem{friedrich2000extracting}
Rudolf Friedrich, Silke Siegert, Joachim Peinke, Marcus Siefert, Michael
  Lindemann, Jan Raethjen, G{\"u}ntner Deuschl, Gerhard Pfister, et~al.
\newblock Extracting model equations from experimental data.
\newblock {\em Physics Letters A}, 271(3):217--222, 2000.

\bibitem{frishman2020learning}
Anna Frishman and Pierre Ronceray.
\newblock Learning force fields from stochastic trajectories.
\newblock {\em Physical Review X}, 10(2):021009, 2020.

\bibitem{kaheman2020sindy}
Kadierdan Kaheman, J~Nathan Kutz, and Steven~L Brunton.
\newblock Sindy-pi: a robust algorithm for parallel implicit sparse
  identification of nonlinear dynamics.
\newblock {\em Proceedings of the Royal Society A}, 476(2242):20200279, 2020.

\bibitem{kaiser2018sparse}
Eurika Kaiser, J~Nathan Kutz, and Steven~L Brunton.
\newblock Sparse identification of nonlinear dynamics for model predictive
  control in the low-data limit.
\newblock {\em Proceedings of the Royal Society A}, 474(2219):20180335, 2018.

\bibitem{keller2021discovery}
Rachael~T Keller and Qiang Du.
\newblock Discovery of dynamics using linear multistep methods.
\newblock {\em SIAM Journal on Numerical Analysis}, 59(1):429--455, 2021.

\bibitem{khasminskii2011stochastic}
Rafail Khasminskii.
\newblock {\em Stochastic stability of differential equations}, volume~66.
\newblock Springer Science \& Business Media, 2011.

\bibitem{kloeden1992stochastic}
Peter~E Kloeden and Eckhard Platen.
\newblock Numerical solution of stochastic differential equations.
\newblock In {\em Numerical solution of stochastic differential equations},
  pages 103--160. Springer, 1992.

\bibitem{kosmatopoulos1995high}
Elias~B Kosmatopoulos, Marios~M Polycarpou, Manolis~A Christodoulou, and
  Petros~A Ioannou.
\newblock High-order neural network structures for identification of dynamical
  systems.
\newblock {\em IEEE transactions on Neural Networks}, 6(2):422--431, 1995.

\bibitem{kumpati1990identification}
S~Narendra Kumpati, Parthasarathy Kannan, et~al.
\newblock Identification and control of dynamical systems using neural
  networks.
\newblock {\em IEEE Transactions on neural networks}, 1(1):4--27, 1990.

\bibitem{ljung1998system}
L.~Ljung.
\newblock {\em System Identification: Theory for the User}.
\newblock Pearson Education, 1998.

\bibitem{mangan2016inferring}
Niall~M Mangan, Steven~L Brunton, Joshua~L Proctor, and J~Nathan Kutz.
\newblock Inferring biological networks by sparse identification of nonlinear
  dynamics.
\newblock {\em IEEE Transactions on Molecular, Biological and Multi-Scale
  Communications}, 2(1):52--63, 2016.

\bibitem{messenger2021weak2}
Daniel~A Messenger and David~M Bortz.
\newblock Weak sindy for partial differential equations.
\newblock {\em Journal of Computational Physics}, 443:110525, 2021.

\bibitem{messenger2021weak}
Daniel~A Messenger and David~M Bortz.
\newblock Weak sindy: Galerkin-based data-driven model selection.
\newblock {\em Multiscale Modeling \& Simulation}, 19(3):1474--1497, 2021.

\bibitem{mezic2005spectral}
Igor Mezi{\'c}.
\newblock Spectral properties of dynamical systems, model reduction and
  decompositions.
\newblock {\em Nonlinear Dynamics}, 41:309--325, 2005.

\bibitem{michalet2012optimal}
Xavier Michalet and Andrew~J Berglund.
\newblock Optimal diffusion coefficient estimation in single-particle tracking.
\newblock {\em Physical Review E}, 85(6):061916, 2012.

\bibitem{qian1991single}
Hong Qian, Michael~P Sheetz, and Elliot~L Elson.
\newblock Single particle tracking. analysis of diffusion and flow in
  two-dimensional systems.
\newblock {\em Biophysical journal}, 60(4):910--921, 1991.

\bibitem{ragwitz2001indispensable}
Mario Ragwitz and Holger Kantz.
\newblock Indispensable finite time corrections for fokker-planck equations
  from time series data.
\newblock {\em Physical Review Letters}, 87(25):254501, 2001.

\bibitem{reiersol1945confluence}
Olav Reiers{\o}l.
\newblock {\em Confluence analysis by means of instrumental sets of variables}.
\newblock PhD thesis, Almqvist \& Wiksell, 1945.

\bibitem{schmid2010dynamic}
Peter~J Schmid.
\newblock Dynamic mode decomposition of numerical and experimental data.
\newblock {\em Journal of fluid mechanics}, 656:5--28, 2010.

\bibitem{sicard2021position}
Francois Sicard, Vladimir Koskin, Alessia Annibale, and Edina Rosta.
\newblock Position-dependent diffusion from biased simulations and markov state
  model analysis.
\newblock {\em Journal of Chemical Theory and Computation}, 17(4):2022--2033,
  2021.

\bibitem{siegert1998analysis}
Silke Siegert, R~Friedrich, and Joachim Peinke.
\newblock Analysis of data sets of stochastic systems.
\newblock {\em Physics Letters A}, 243(5-6):275--280, 1998.

\bibitem{soderstrom1983instrumental}
Torsten. S{\"o}derstr{\"o}m and Petre. Stoica.
\newblock {\em Instrumental variable methods for system identification}.
\newblock Lecture notes in control and information sciences ; 57.
  Springer-Verlag, Berlin ;, 1983.

\bibitem{tibshirani1996regression}
Robert Tibshirani.
\newblock Regression shrinkage and selection via the lasso.
\newblock {\em Journal of the Royal Statistical Society: Series B
  (Methodological)}, 58(1):267--288, 1996.

\bibitem{vestergaard2014optimal}
Christian~L Vestergaard, Paul~C Blainey, and Henrik Flyvbjerg.
\newblock Optimal estimation of diffusion coefficients from single-particle
  trajectories.
\newblock {\em Physical Review E}, 89(2):022726, 2014.

\bibitem{wanner2022robust}
Mathias Wanner and Igor Mezic.
\newblock Robust approximation of the stochastic koopman operator.
\newblock {\em SIAM Journal on Applied Dynamical Systems}, 21(3):1930--1951,
  2022.

\bibitem{williams2015data}
Matthew~O Williams, Ioannis~G Kevrekidis, and Clarence~W Rowley.
\newblock A data--driven approximation of the koopman operator: Extending
  dynamic mode decomposition.
\newblock {\em Journal of Nonlinear Science}, 25:1307--1346, 2015.

\bibitem{zhang2019convergence}
Linan Zhang and Hayden Schaeffer.
\newblock On the convergence of the sindy algorithm.
\newblock {\em Multiscale Modeling \& Simulation}, 17(3):948--972, 2019.

\end{thebibliography}

\newpage
\appendix
\section{Error Derivations for Section \ref{sec:FirstOrder}}

In Theorems \ref{th:DriftOrd1}-\ref{th:DiffTrap}, we used estimates of the drift and diffusion based on finite differences. Most of the derivations are straightforward and follow almost immediately from the Ito-Taylor expansions. However, bounding the estimate for the variance in the second order difference for the drift and the first order estimate for the diffusion require a little extra work, so we include them here.
\subsection{Drift: Second Order Forward Difference}
\label{App:ErrorDrift1}
The error in the second order forward difference estimate (\ref{eq:FDOrd2}) for the drift is given by
\begin{equation*}
e_{t_n}=\mu^i(X_{t_n})-\frac{-3X^i_{t_n}+4X^i_{t_{n+1}}-X^i_{t_{n+2}}}{2\Delta t}.
\end{equation*}
Using (\ref{eq:WeakDrift}) in the estimate above gives us
\begin{equation}
\label{eq:AppDriftErrOrd2}
\E(e_t\,|\,X_t)=-\frac{\Delta t}{3}(L^0)^2\mu^i +O(\Delta t^3).
\end{equation}
\begin{proof}[Proof of Theorem \ref{th:DriftOrd2}]
The proof of the estimate on the mean error follows from (\ref{eq:AppDriftErrOrd2}) and the proof of Theorem \ref{th:DriftOrd1}. 
Now, Let
\[\Theta_0=\begin{bmatrix}
\theta(X_{t_0}) \\ \theta(X_{t_{1}}) \\ \vdots \\ \theta(X_{t_{N-1}})
\end{bmatrix}~~~~\text{and}~~~~E=\begin{bmatrix} e_0 \\ e_1 \\ \vdots \\ e_{N-1} \end{bmatrix}.\]
To estimate the variance, we need to find $\E(\|\frac{1}{N}\Theta_0^*E\|_2^2)$. To do this, we will use the strong expansion (\ref{eq:StrongDrift}) and obtain
\[e_t=\frac{1}{2\Delta t}\left(\sum_{m=1}^d\sigma^{i,m}(3\Delta W^m_t-\Delta W^m_{t+1})+R_t\right)\]
with $\E(|R_t|^2)=O(\Delta t^2)$.
Then, using the 
\begin{align*}
\frac{1}{N}\Theta_0^*E&=\frac{1}{N}\sum_{n=0}^{N-1} \theta^*(X_{t_n})e_{t_n}=\frac{1}{N}\sum_{n=0}^{N-1} \theta^*(X_{t_n})\left(\sum_{m=1}^d \sigma^{i,m}(X_{t_n})\frac{3\Delta W^m_{t_n}-\Delta W^m_{t_{n+1}}}{2\Delta t}+\frac{R_{t_n}}{2\Delta t}\right)\\
&=\frac{1}{2T}\sum_{n=0}^{N-1} \theta^*(X_{t_n})\left(\sum_{m=1}^d (3\sigma^{i,m}(X_{t_n})-\sigma^{i,m}(X_{t_{n-1}}))\Delta W^m_{t_n}+R_{t_n}\right)+R_1\\
&=\frac{1}{T}\sum_{n=0}^{N-1} \theta^*(X_{t_n})\left(\sum_{m=1}^d (\sigma^{i,m}(X_{t_n})+R^m_{t_n})\Delta W^m_{t_n}+R_{t_n}\right)+R_1
\end{align*}
where
\[R_1=\frac{1}{T}\sum_{m=1}^d\left(\theta^*(X_{t_0})\sigma^{i,m}(X_{t_0})\Delta W^m_{t_0}-\theta^*(X_{t_N})\sigma^{i,m}(X_{t_N})\Delta W^m_{t_N}\right)\] and $\E(|R^m_{t_n}|^2)=O(\Delta t)$. The second line comes from rearranging the indices of the sum which gives the remainder $R_1$ and he last line uses the Ito-Taylor expansion of $\sigma^{i,m}$, which gives the remainder $R_2$. Combining all of  the errors gives us
\[\frac{1}{N}\Theta_0^*E=\frac{1}{T}\sum_{n=0}^{N-1}\sum_{m=1}^d \theta^*(X_{t_n})\sigma^{i,m}(X_{t_n})\Delta W^m_{t_n}+R\]
with $\E(R^2)=O(\Delta t^2)$. Taking the expectance of the square of this last equation gives us
\[\E\left(\left\|\frac{1}{N}\Theta_0^*E\right\|_2^2\right)\leq\frac{1}{T^2}\sum_{n=0}^N\sum_{m=1}^d\|\theta^*(X_{t_n})\|_2^2 \sigma^{i,m}(X_{t_n})^2\Delta t+O(\Delta t^{\frac{3}{2}}).\]
Using this, the rest of the proof follows that of Theorem \ref{th:DriftOrd1}. 
\end{proof}

\subsection{Diffusion: First Order Forward Difference}
\label{App:ErrDiff1}
For Theorem \ref{th:DiffOrd1}, we use (\ref{eq:FDsquared}) to approximate the diffusion matrix elements. We need to bound the errors to give equations (\ref{eq:DiffErrm}) and (\ref{eq:DiffErrsq}) for the proof. From the approximation (\ref{eq:FDsquared}), the error is
\[e_t=\frac{(X^i_{t+\Delta t}-X^i_{t})(X^j_{t+\Delta t}-X^j_{t})}{2\Delta t}-\Sigma^{i,j}(X_t).\]
The expected error, $\E(e_t|X_t)$ is easy to bound using (\ref{eq:WeakDiff}). To calculate the squared error, \newline $\E(|e_t|^2|X_t)$, we will use the strong expansion (\ref{eq:StrongDiff}). This gives us
\begin{equation}
\label{eq:AppDiffErr}
e_t=\frac{1}{2\Delta t}\left(\sum_{k,l=1}^d (\sigma^{k,i}\sigma^{l,j}(X_t)+\sigma^{k,j}\sigma^{l,i}(X_t))I_{i,j}+R_t\right)
\end{equation}
with
$\E(|R_t|^2|X_t)=O(\Delta t^3)$. From Lemma 5.7.2 of \cite{kloeden1992stochastic}, we have
\[\E(I_{(k,l)}I_{(m,n)})=\begin{cases} 0, & (k,l)\neq(m,n) \\ \frac{\Delta t^2}{2}, & k=m, l=n.\end{cases}\]
Then, squaring (\ref{eq:AppDiffErr}), we get
\begin{align*}
\E(|e_t|^2\,|\,X_t)&=\frac{1}{4\Delta t^2}\sum_{k,l=1}^d(\sigma^{k,i}\sigma^{l,j}(X_t)+\sigma^{k,j}\sigma^{l,i}(X_t))^2\frac{\Delta t^2}{2}+O(\Delta t^{\frac{1}{2}})\\
&=\frac{1}{8}\sum_{k,l=1}^d 2\left((\sigma^{k,i}\sigma^{l,j}(X_t))^2+\sigma^{k,i}\sigma^{l,j}\sigma^{k,j}\sigma^{l,i}(X_t)\right)+O(\Delta t^{\frac{1}{2}})\\
&=\Sigma^{i,i}(X_t)\Sigma^{j,j}(X_t)+\Sigma^{i,j}(X_t)^2+O(\Delta t^{\frac{1}{2}}).
\end{align*}
Which gives us equation (\ref{eq:DiffErrsq}).

\subsection{Trapezoidal Approximation for Diffusion}
\label{sec:AppDiffTrap}
The trapezoidal method to approximate the diffusion is given by equation (\ref{eq:TrapSig}):
\begin{equation}
\label{eq:AppTrapSig}
\Sigma^{i,j}(X_{t+\Delta t})+\Sigma^{i,j}(X_t)\approx\frac{\left(\Delta X^i_t-\frac{\Delta t}{2}(\mu^i(X_{t})+\mu^i(X_{t+\Delta t}))\right)\left(\Delta X^j_t-\frac{\Delta t}{2}(\mu^j(X_t)+\mu^j(X_{t+\Delta t}))\right)}{\Delta t}.
\end{equation}
We claim that the error in this approximation can be bounded by
\[|\E(e_t\,|\,X_t)|=O(\Delta t^2)~~~~~\text{and}~~~~~\E(|e_t|^2)=O(\Delta t).\]
To achieve this, we let $\Delta \mu^i_t=\mu^i(X_{t+\Delta t})-\mu^i(X_t)$ and rewrite the right hand side of (\ref{eq:AppTrapSig}) as
\begin{equation}
\label{eq:AppTrapSig2}
\frac{\left(\Delta X^i_t-\mu^i(X_{t})\Delta t-\frac{\Delta t}{2}\Delta \mu^i\right)\left(\Delta X^j_t-\mu^j(X_t)\Delta t-\frac{\Delta t}{2}\Delta \mu^j_t\right)}{\Delta t}=\Sigma^{i,j}(X_t)+\Sigma^{i,j}(X_{t+\Delta t})+e_t.
\end{equation}
We will look at several of the cross terms separately on the left hand side. Using equations (\ref{eq:WeakDiff}) and (\ref{eq:WeakDrift}), we can see that the first term will be
\[\E\left(\frac{\left(\Delta X_t^i-\mu^i(X_t)\Delta t\right)\left(\Delta X_t^j-\mu^j(X_t)\Delta t\right)}{\Delta t}\Big|X_t\right)=2\Sigma^{i,j}(X_t)+h(X_t)\Delta t +O(\Delta t^2),\]
where 
\[h=L^0\Sigma^{i,j}+\sum_{k=1}^d\left(\Sigma^{j,k}\frac{\partial \mu^i}{\partial x^k}+\Sigma^{i,k}\frac{\partial \mu^j}{\partial x^k}\right).\]
Next we will consider the $\Delta X^i_t\Delta \mu^j_t$ terms. If we use the weak Ito-Taylor expansion of $f(x)=(x^i-X^i_t)(\mu^j(x)-\mu^j(X_t))$, holding $X_t$ fixed, we see that
\[\E\left(\Delta X^i_t\Delta \mu^j_t|X_t\right)=\E(f(X_{t+\Delta t})|X_t)=L^0 f(X_t)\Delta t+O(\Delta t^2)=2\sum_{k=1}^d\Sigma^{i,k}\frac{\partial \mu^j}{\partial x^k}.\]
These will cancel the last terms in $h$. All other terms on the left hand side will be of higher order. For the right hand side of (\ref{eq:AppTrapSig2}), we can use the Ito-Taylor expansion of $\Sigma^{i,j}$ to show
\[\E\left(\Sigma^{i,j}(X_{t+\Delta t})|X_t\right)=\Sigma^{i,j}(X_t)+L^0\Sigma^{i,j}\Delta t+O(\Delta t^2).\]
Combining all of these together, we see that
\[\Sigma^{i,j}(X_t)+\Sigma^{i,j}(X_{t+\Delta t})=\frac{\left(\Delta X^i_t-\mu^i(X_{t})\Delta t-\frac{\Delta t}{2}\Delta \mu^i\right)\left(\Delta X^j_t-\mu^j(X_t)\Delta t-\frac{\Delta t}{2}\Delta \mu^j_t\right)}{\Delta t}+e_t\]
with $\E(e_t|x_t)=O(\Delta t^2)$. The bound for the squared error, $\E(|e_t|^2\,|\, X_t)=O(\Delta t)$, follows easily from (\ref{eq:StrongDiff}), since the correction terms added are all of higher order.

\end{document}